\documentclass[12pt,reqno]{amsproc}

\usepackage{latexsym} 
  \usepackage[all]{xy}
  \usepackage{amsfonts} 
  \usepackage{amsthm} 
  \usepackage{amsmath} 
  \usepackage{amssymb}
  \usepackage{pifont}  
  \usepackage{enumerate}
 \newtheorem{thm}{Theorem}[section]
 \newtheorem{cor}[thm]{Corollary}
 \newtheorem{lem}[thm]{Lemma}
 \newtheorem{prop}[thm]{Proposition}
 \theoremstyle{definition}
 \newtheorem{defn}[thm]{Definition}
 \theoremstyle{remark}
 \newtheorem{rem}[thm]{Remark}
 \newtheorem{ex}[thm]{Example}
 \numberwithin{equation}{section}
 
 \def\bs{\begin{statement}}
\def\es{\end{statement}}
  \swapnumbers
  \newtheorem{statement}[thm]{} 
  
   \def\sw#1{{\sb{(#1)}}} 
    \def\eps{\varepsilon}

  \newcounter{zlist} 
  \newenvironment{zlist}{\begin{list}{(\arabic{zlist})}{ 
  \usecounter{zlist}\leftmargin2.5em\labelwidth2em\labelsep0.5em 
  \topsep0.6ex
  \parsep0.3ex plus0.2ex minus0.1ex}}{\end{list}}

  \newcounter{blist}

  \newcounter{rlist} 
  \newenvironment{rlist}{\begin{list}{(\roman{rlist})}{ 
  \usecounter{rlist}\leftmargin2.5em\labelwidth2em\labelsep0.5em 
  \topsep0.6ex 
  \parsep0.3ex plus0.2ex minus0.1ex}}{\end{list}}

\def\ot{\otimes}

\def\FF{{\mathbb F}}

\def\KK{{\mathbb K}}

\def\NN{{\mathbb N}}

\newcommand{\Bb}{\mathcal{B}}
\newcommand{\Cc}{\mathcal{C}}

\newcommand{\Tt}{\mathcal{T}}

\def\*C{{}^*\hspace*{-1pt}{\Cc}}

\def\text#1{{\rm {\rm #1}}}

 \def\1{\mathbf{1}}

  \def\la{\triangleright}
 \def\bla{\!\blacktriangleright\!}
   
\def\id{\mathrm{id}}
\def\pi {\mathrm{pi}}
\def\endeg{\flushright $\Diamond$}

\begin{document}

\title{Trusses: between braces and rings}

\author{Tomasz Brzezi\'nski}

\address{
Department of Mathematics, Swansea University, 
Swansea University Bay Campus,
Fabian Way,
Swansea,
SA1 8EN, U.K.\ \newline \indent
Department of Mathematics, University of Bia{\l}ystok, K.\ Cio{\l}kowskiego  1M,
15-245 Bia\-{\l}ys\-tok, Poland}

\email{T.Brzezinski@swansea.ac.uk}



\begin{abstract}
In an attempt to understand the origins and the nature of the law binding  two group operations together into a {\em skew brace}, introduced in [L.\ Guarnieri \& L.\ Vendramin,  Math.\ Comp.\ \textbf{86} (2017), 2519--2534] as a non-Abelian version of the {\em brace distributive law} of [W.\ Rump,  
J.\ Algebra {\bf 307} (2007), 153--170] and [F.\ Ced\'o, E.\ Jespers \& J.\ Okni\'nski,  Commun.\ Math.\ Phys.\ {\bf 327} (2014), 101--116], the notion of a {\em skew truss} is proposed. A skew truss consists of a set with a group operation and a semigroup operation connected by a modified distributive law that interpolates between that of a ring and a brace. It is shown that  a particular action and a cocycle characteristic of skew braces are already present  in a skew truss; in fact the interpolating function is a 1-cocycle, the bijecitivity of which indicates the existence of an operation that turns a truss into a brace. Furthermore, if the group structure in a two-sided truss is Abelian, then there is an associated ring -- another feature characteristic of a two-sided brace. To characterise a morphism of trusses, a {\em pith} is defined as a particular subset of the domain consisting of subsets termed {\em chambers}, which contains the kernel of the morphism as a group homomorphism. In the case of both rings and braces piths coincide with kernels. In general the pith of a morphism is a sub-semigroup of the domain and, if additional properties are satisfied, a pith is an $\NN_+$-graded semigroup. Finally, giving heed to [I.\ Angiono, C.\ Galindo \& L.\ Vendramin, Proc.\ Amer.\ Math.\ Soc.\ {\bf 145} (2017), 1981--1995] we linearise trusses and thus define {\em Hopf trusses} and study their properties, from which, in parallel to the set-theoretic case, some properties of Hopf braces are shown to follow.
\end{abstract}

\maketitle

\section{Introduction}
The quantum Yang-Baxter equation has its origins in statistical mechanics and quantum field theory \cite{Yan:som}, \cite{Bax:par}, \cite{Bax:exa}, but it came to prominence with the development of Hamiltonian methods in the theory of integrable systems \cite{FadTak:Ham}, \cite{Jim:qua}, leading to the introduction of quantum groups \cite{Dri:qua}. Since the late 1980s the quantum Yang-Baxter or braid equation has been the subject of numerous studies both in mathematical physics and pure mathematics. The study of a set-theoretic variant of the Yang-Baxter equation was proposed by Drinfeld in \cite{Dri:uns}. Though attracting a rather limited attention at the beginning, notwithstanding the ground-breaking contributions such as in \cite{GatVan:sem}, \cite{EtiSch:set}, \cite{LuYan:set}, the set theoretic Yang-Baxter equation has become a subject of intensive study more recently after the discovery of richness of the algebraic structure of sets that admit involutive, non-degenerate solutions to the Yang-Baxter equation, and its impact on group theory. We comment on this richness of structure presently and in more detail, as it is not the history or the development, neither the applications nor solutions of the Yang-Baxter equation that form the subject matter of the paper the reader is hereby presented with. This text is motivated by the emergence of algebraic systems known as {\em braces} \cite{Rum:dec}, \cite{Rum:bra}, \cite{CedJes:bra} and {\em skew braces} \cite{GuaVen:ske}, which are tantamount to the existence of non-degenerate solutions of the set-theoretic Yang-Baxter equation. A (skew) brace is a set with two binary operations related to each other by a specially adapted distributive law. More specifically a set $A$ together with two binary operations $\diamond$ and $\circ$ each one making $A$ into a group is called a {\em skew left brace} if, for all $a,b,c\in A$,
\begin{equation}\label{brace.law}
a\circ (b\diamond c) = (a\circ b) \diamond a^\diamond\diamond(a\circ c),
\end{equation}
where $a^\diamond$ is the inverse of $a$ with respect to $\diamond$. We refer to the compatibility condition \eqref{brace.law} as to the {\em brace distributive law}. A skew left brace is called simply a {\em left brace}, provided $(A,\diamond)$ is an Abelian group. 

Since their introduction and through their connection with the Yang-Baxter equation and, more fundamentally, group theory, braces attracted a lot of attention and the literature devoted to them and their generalisations is growing rapidly; see e.g.\ \cite{CatCol:sem}--\cite{Gat:set} to mention but a few. The questions the author of these notes would like to address are these: what is specific about the combination of operations that make the set into a brace and what properties of the braces are a (necessary) result of a more general compatibility relation?

In order to answer the said questions we propose to study a set with two binary operations connected by a rule which can be seen as the interpolation between the ring-type (i.e.\ the standard) and brace distributive laws. We call such a system a {\em skew truss}.\footnote{One might think about the proposed law as a way of holding different structures together,  hence the name. For the sake of simplicity of exposition we often drop the adjective `skew' for the rest of this introduction.} More exactly, a skew left truss is a set $A$ with binary operations $\diamond$ and $\circ$  and a function $\sigma: A\to A$, such that $(A,\diamond)$ is a group, $(A,\circ)$ is a semigroup\footnote{Note that the weakening of the structures that enter a truss goes in the different direction than is the case of a {\em semi-brace} \cite{CatCol:sem}, where $(A,\diamond)$ is a (left cancellative) semigroup, while $(A,\circ)$ is a group.} and, for all $a,b,c\in A$, the following generalised distributive law holds
\begin{equation}\label{truss.law}
a\circ (b\diamond c) = (a\circ b) \diamond \sigma(a)^\diamond\diamond(a\circ c).
\end{equation}
 This {\em truss distributive law} interpolates the ring (standard) and brace distributive laws: the former one is obtained by setting $\sigma(a) =1_\diamond$ (the neutral element of $(A,\diamond)$), the latter is obtained by setting $\sigma$ to be the identity map. 
 
 Although \eqref{truss.law} might seem as an {\it ad hoc} modification of both ring and brace distributive laws it can be equivalently written in a number of ways: 
$$
a\circ(b\diamond c) = (a\circ b)\diamond \lambda_a(c),\quad \!
a\circ(b\diamond c) = \mu_a(b)\diamond (a\circ c) , \quad\! \mbox{or} \quad\!
a\circ (b\diamond c) = \kappa_a(b)\diamond \hat\kappa_a(c),
$$
where, for each $a\in A$,  $\lambda_a$, $\mu_a$, $\kappa_a$ and $\hat\kappa_a$ are (unique) endomaps on $A$ (see Theorem~\ref{thm.equiv}). In other less formal words, if one considers a general sensible modification of the distributive law between operations $\diamond$ and $\circ$ on $A$ (including one involving two seemingly unrelated new binary operations given by $\kappa$ and $\hat\kappa$) one is forced to end up with the truss distributive law \eqref{truss.law}.

In fact the truss distributive law describes the distributivity of $\circ$ over the ternary {\em heap} or {\em torsor} or {\em herd}  operation $[a,b,c] = a\diamond b^\diamond\diamond c$  induced by $\diamond$;\footnote{I am grateful to Michael Kinyon for pointing this out to me.} see \cite[page 170]{Pru:the},  \cite[page 202, footnote]{Bae:ein} or \cite[Definition~2]{Cer:ter} for the definition of a herd or heap. Then the truss distributive law \eqref{truss.law} is equivalent to the distributive law
$$
a\circ [b,c,d] = [a\circ b, a\circ c, a\circ d], \qquad \mbox{for all $a,b,c,d\in A$};
$$
see Theorem~\ref{thm.equiv}. This not only gives additional justification for the truss distributive law, but also suggests immediate extension of the notion of a truss beyond groups to general heaps (or torsors or herds).\footnote{See, however, Lemma~\ref{lem.heap} which can somewhat dampen one's enthusiasm.}

There are properties of (skew) braces that are consequences of the truss distributive law. For example, \eqref{truss.law} implies that  the semigroup $(A,\circ)$ acts on the group $(A,\diamond)$ by group endomorphisms in two different ways (see Theorem~\ref{thm.brace}), thus yielding similar statements for braces \cite{Rum:bra}  and skew braces \cite[Proposition~1.9]{GuaVen:ske}. Furthermore, the map $\sigma$ in the  truss distributive law \eqref{truss.law} satisfies a 1-cocycle condition (see Theorem~\ref{thm.brace}), which provides a part of the linkage between bijective group 1-cocycles and braces \cite{Rum:bra} and skew braces \cite[Proposition~1.11]{GuaVen:ske}. Further still, if $(A,\diamond)$ is an Abelian group, and both left and, symmetrically defined, right truss distributive laws  with the same cocycle $\sigma$ are satisfied, then a new associative operation on $A$ can be defined, which combined with $\diamond$ make $A$ into a ring; see Theorem~\ref{thm.brace.ring}.  A part of the connection between braces and radical rings revealed in \cite{Rum:mod} follows from this. Finally, if the cocycle is bijective, then the semigroup operation can be ported  through it so that the original group and the new semigroup operations are connected by the brace distributive law. In particular, there is a brace associated to any truss with both operations making the set into a group, although this brace is usually not isomorphic to the truss it originates from; see Lemma~\ref{lem.brace}. This observation shows the limitation of the freedom in defining the truss structure, and most importantly indicates that if a set has two group operations, {\bf\em the only sensible distributivity one can expect is that of a brace}: the brace law is both rigid and unique.\footnote{This {\em rigidity property} has a flavour of the Eckmann-Hilton argument \cite{EckHil:gro} stating that if a set admits two unital operations commuting with each other, then these operations are necessarily equal to each other, and they are commutative and associative.} 

Yet there are properties of trusses that make them distinct from braces (and rings). If a group and monoid operations on the same set are connected by the brace distributive law, then necessarily the neutral elements for both these structures must coincide. This is no longer the case for trusses; it is the cocycle that transforms the monoid identity into the group identity (Corollary~\ref{cor.sigma.inverse}). There is no need for the neutral element of the group to act trivially on other elements of the set by the semigroup operation, as is the case for rings, either.

Once a single truss is constructed there is a possibility of constructing a (possibly infinite) family of trusses that includes those 
in which the cocycles are given by self-compositions of the original 1-cocycle; see Corollary~\ref{cor.family}. Members of this family generally are not mutually isomorphic. If both operations making a set into a truss are group operations, and hence the cocycle is bijective, each member of this (infinite, if the cocycle has infinite order) hierarchy can be transformed into a brace. Whether this observation is functional in constructing new solutions to the Yang-Baxter equation only time will tell. 

Although it is clear how to define a morphism of trusses, it is no longer so clear, how to define useful algebraic systems such as kernels that characterise morphisms to a certain extent. A morphism of trusses is in particular a morphism of the group structures, but the group-theoretic kernel of the truss morphism turns out to be too small and too much disconected from the semigroup structure to carry sufficient amount of useful information. This is in contrast to braces, in which the fact that neutral elements for both structures necessarily coincide, indicates in a very suggestive way, what the kernel of a brace morphism should be. Instead of a single set associated to a truss morphism $f$, we construct a family of such sets, each one controlling the behaviour of $f$ with the group identity and powers of the cocycle. The union of all these sets called {\em chambers} forms the {\em pith} of $f$.\footnote{We tend to imagine the pith of a morphism as, say, the pith of a walnut shoot, i.e.\ as central, consisting of chambers and containing the kernel; or as a pith of an orange, which can be discarded.} The pith of $f$ is a sub-semigroup of the domain, the labelling of chambers is compatible with the semigroup operation, the cocycle sends elements of each chamber to the next one; see Lemma~\ref{lem.kernel}. Furthermore, the zeroth chamber of the pith of $f$ euqals to the kernel of $f$ as a group morphism, and hence is a normal subgroup of the domain of $f$.

Motivated by Angiono,  Galindo and  Vendramin, who in \cite{AngGal:Hop} linearised skew braces into Hopf braces, we linearise skew trusses into Hopf trusses in Section~\ref{sec.Hopf}. Thus a Hopf truss consists of a coalgebra with two algebra structures such that one makes it into a Hopf algebra and the other into a bialgebra, connected by a linear version of the truss distributive law \eqref{distribute.Hopf}. Many similarities to or differences from as well as limitations of skew trusses described above carry over to the Hopf case. In particular, the linear version of the truss distributive law must involve Hopf actions and a cocycle (Theorem~\ref{thm.truss.Hopf}), there is a hierarchy of mutually non-isomorphic Hopf trusses (Proposition~\ref{prop.hierarchy.Hopf}), and if both algebras are Hopf algebras, then there is a Hopf brace associated to the Hopf truss (Lemma~\ref{lem.Hopf.brace}). From the categorical point of view the occurrence of these similarities and differences is not at all unexpected. Groups can be understood as Hopf algebras within a symmetric monoidal category of sets (with the Cartesian product as the monoidal structure), hence it is natural to expect similar behaviour in all situations in which this categorical nature of both Hopf algebras and groups takes a dominant role, as is the case of trusses.

\subsection*{Notation.} In any monoid, say, $(A,\diamond)$, the neutral element is denoted by $1_\diamond$. The inverse of an element $a\in A$ with respect to $\diamond$ and $1_\diamond$ is denoted by $a^\diamond$. For any sets $A,B$, $\mathrm{Map}(A,B)$ denotes the set of all functions from $A$ to $B$. For any $f\in \mathrm{Map}(A,B)$, the {\em push-forward} and {\em pullback} functions are defined and denoted as
$$
f_* : \mathrm{Map}(A,A) \to \mathrm{Map}(A,B), \qquad \varphi \mapsto f\circ \varphi,
$$
and
$$
f^* : \mathrm{Map}(B,B) \to \mathrm{Map}(A,B), \qquad \varphi \mapsto  \varphi \circ f,
$$
respectively. The space of all homomorphisms from an $\FF$-vector space $A$ to an $\FF$-vector space $B$ is denoted by $\mathrm{Hom}_ \FF (A, B)$, and unadorned tensor product is over $\FF$. The composition of functions is denoted by juxtposition.

  In Section~\ref{sec.Hopf} the standard Hopf algebra notation is used: the Sweedler notation for the value of a comultiplication $\Delta$ on elements, $\eps$ for a counit and $S$ for an antipode. 

\section{Skew trusses: definition and properties}\label{sec.truss}
In this section we define trusses and their morphisms, describe their equivalent characterisations, and study their properties.
\begin{defn}\label{def.truss}
Let  $(A,\diamond)$ be a group and let $\circ$ be an associative operation on $A$, for which there exists a function $\sigma:A\to A$ 
such that, for all $a,b,c\in A$,
\begin{equation}\label{semi.brace}
a\circ (b\diamond c) = (a\circ b) \diamond \sigma(a)^\diamond\diamond(a\circ c).
\end{equation}
A quadruple $(A,\diamond,\circ,\sigma)$,  is called a {\em skew left truss}, and the function $\sigma$ is called a {\em cocycle}.

In a symmetric way, a {\em skew right truss} is a quadruple $(A,\diamond,\circ,\sigma)$ with operations $\diamond$, $\circ$ making $A$ into a group and a semigroup respectively, and $\sigma: A\to A$ such that
\begin{equation}\label{distribute.r}
(a\diamond b)\circ c = (a\circ c)\diamond \sigma(c)^\diamond \diamond(b\circ c),
\end{equation}
for all $a,b,c\in A$.

The adjective {\em skew} is dropped if the group $(A,\diamond)$ is abelian. 

The identities \eqref{semi.brace}  and  \eqref{distribute.r} are refered to as {\em left} and {\em right truss dristributive laws}, respectively.
\end{defn}

Motivating examples of skew left trusses are a skew left ring  or a left near-ring (in which case $\sigma(a) = 1_\diamond$) and a skew left brace (in which case  $\sigma(a)  = a$,  and it is additionally assumed that $(A,\circ)$ is a group). 

\begin{ex}\label{ex.truss}
Let $(A,\diamond)$ be a group and let $\sigma\in \mathrm{Map}(A,A)$ be an idempotent function. Define a binary operation $\circ$ on $A$ by
\begin{equation}\label{idempotent}
a\circ b = \sigma(a), \qquad \mbox{for all $a,b\in A$}.
\end{equation}
Then $(A,\diamond,\circ)$ is a skew left truss with cocycle $\sigma$.
\end{ex}
\begin{proof}
Since $\sigma$ is an idempotent function, i.e.\ $\sigma^2 = \sigma$, $(A,\circ)$ is a semigroup. Furthermore,
$$
(a\circ b)\diamond \sigma(a)^\diamond \diamond (a\circ c) = \sigma(a)\diamond \sigma(a)^\diamond \diamond \sigma(a) = \sigma(a) = a\circ (b\diamond c),
$$
hence the operations are related by \eqref{semi.brace}, as required.
\end{proof}

Whatever is said about a skew left truss can be symmetrically said about a skew right truss. Therefore we can focus on the former, and we skip the adjective `left': the phrase `skew truss' means `skew left truss'. 

In a skew truss the cocyle is fully determined by the neutral element of the group and by the semigroup operation. 
\begin{lem}\label{lem.coc.prop}
Let $(A,\diamond,\circ,\sigma)$ be a skew truss. Then, for all $a\in A$,
\begin{equation}\label{sigma}
\sigma( a)= a\circ 1_\diamond.
\end{equation}
Consequently, the cocycle $\sigma$ is $(A,\circ)$-equivariant in the sense that, for all $a,b\in A$,
\begin{equation}\label{equivariant}
\sigma(a\circ b) = a\circ \sigma(b).
\end{equation}
\end{lem}
\begin{proof}
Set $b=c=1_\diamond$ in the law \eqref{semi.brace} to obtain
$$
a\circ 1_\diamond = a\circ(1_\diamond\diamond1_\diamond) = (a\circ 1_\diamond )\diamond \sigma(a)^\diamond \diamond (a\circ 1_\diamond).
$$ 
The cancellation laws in $(A,\diamond)$ yield \eqref{sigma}. 

The second statement follows by the associativity of $\circ$ and \eqref{sigma}.
\end{proof}

 Directly from the truss distributive law \eqref{semi.brace} one derives the following properties:
 \begin{lem}\label{lem.diamond.inv}
Let $(A,\diamond,\circ,\sigma)$ be a skew truss. Then, for all $a,b,c\in A$,
\begin{subequations}\label{binv}
\begin{equation}\label{binv.c}
a\circ (b^\diamond\diamond c) = \sigma(a)\diamond (a\circ b)^\diamond\diamond (a\circ c), 
\end{equation}
\begin{equation}\label{b.cinv}
a\circ (b\diamond c^\diamond) =  (a\circ b) \diamond (a\circ c) ^\diamond\diamond\sigma(a) . 
\end{equation}
\end{subequations}
\end{lem}
\begin{proof}
To prove \eqref{binv.c}, let us consider
$$
a\circ c = a\circ(b\diamond b^\diamond\diamond c) = (a\circ b)\diamond \sigma(a)^\diamond\diamond(a\circ  (b^\diamond\diamond c)),
$$
where the second equality follows by \eqref{semi.brace}. 
This implies \eqref{binv.c}. 
The equality \eqref{b.cinv} is proven in an analogous way.
\end{proof}

It turns out that in  \eqref{semi.brace} the notion of a skew truss captures general distributivity of a group and semigroup operations defined on the same set.

\begin{thm}\label{thm.equiv}
Let $A$ be a set with two binary operations $\diamond$ and $\circ$ making $A$ into a group and semigroup respectively. Then the following statements are equivalent
\begin{zlist}
\item There exists $\sigma: A\to A$ such that $(A,\diamond,\circ,\sigma)$ is a skew truss.
\item There exists 
$$
\lambda : A \to \mathrm{Map}(A,A), \qquad a\mapsto [\lambda_a: A\to A],
$$
such that, for all $a,b,c\in A$,
\begin{equation}\label{distribute}
a\circ (b\diamond c) = (a\circ b)\diamond \lambda_a(c).
\end{equation}
\item There exists 
$$
\mu : A \to \mathrm{Map}(A,A), \qquad a\mapsto [\mu_a: A\to A],
$$
such that, for all $a,b,c\in A$,
\begin{equation}\label{distribute.mu}
a\circ (b\diamond c) = \mu_a(b)\diamond (a\circ c).
\end{equation}
\item There exist
$
\kappa,\hat\kappa : A \to \mathrm{Map}(A,A),
$
$$
\kappa: a\mapsto [\kappa_a: A\to A], \qquad \hat\kappa: a\mapsto [\hat\kappa_a: A\to A],
$$
such that, for all $a,b,c\in A$,
\begin{equation}\label{distribute.kap}
a\circ (b\diamond c) = \kappa_a(b)\diamond \hat\kappa_a(c).
\end{equation}
\item Denote by $[-,-,-]$ the ternary {\em heap} or {\em torsor} or {\em herd} operation on $(A,\diamond)$, 
$$
[-,-,-]: A\times A\times A\to A, \qquad [a,b,c] := a\diamond b^\diamond \diamond c.
$$
Then $\circ$ left distributes over $[-,-,-]$, i.e.\  for all $a,b,c,d\in A$,
\begin{equation}\label{heap.distribute}
a\circ [b,c,d] = [a\circ b, a\circ c, a\circ d].
\end{equation}
\end{zlist}
\end{thm}
\begin{proof}
Obviously statement (1) implies  statements (2), (3) and (4). 

Assume that (2) holds, and let $\sigma(a) = a\circ 1_\diamond$. Setting $b=1_\diamond$ in \eqref{distribute} we obtain that $a\circ c = \sigma(a)\diamond \lambda_a(c)$. Therefore,
\eqref{distribute} can be rewritten as
$$
a\circ (b\diamond c) = (a\circ b) \diamond \sigma(a)^\diamond \diamond \sigma(a)\diamond \lambda_a(c) = (a\circ b) \diamond \sigma(a)^\diamond \diamond(a\circ c).
$$ 
Hence $(A,\diamond,\circ,\sigma)$ is a skew truss.

Assume that (3) holds. Letting $\sigma (a) = a\circ 1_\diamond$ and setting $c= 1_\diamond$ in \eqref{distribute.mu} we obtain that $a\circ b = \mu_a(b)\diamond \sigma(a)$. Hence  
\eqref{distribute} can be rewritten as
$$
a\circ (b\diamond c) = \mu_a(b)\diamond \sigma(a) \diamond \sigma(a)^\diamond \diamond (a\circ c)  = (a\circ b) \diamond \sigma(a)^\diamond \diamond(a\circ c),
$$ 
and $(A,\diamond,\circ,\sigma)$ is a skew truss, as required.

Finally, let us assume that (4) holds and define $\tau(a) = \hat\kappa_a(1_\diamond)$. Setting $c=1_\diamond$ in \eqref{distribute.kap} we obtain 
$$
a\circ b =  \kappa_a(b)\diamond \tau(a) .
$$
As a result, \eqref{distribute.kap} can be rewritten as
$$
a\circ (b\diamond c) = \kappa_a(b)\diamond \tau(a)\diamond \tau(a)^\diamond \diamond \hat\kappa_a(c) = (a\circ b)\diamond \tau(a)^\diamond \diamond \hat\kappa_a(c).
$$ 
Therefore, \eqref{distribute.kap} takes the form of \eqref{distribute} (with $\lambda_a(c) = \tau(a)^\diamond \diamond \hat\kappa_a(c)$), i.e.\ statement (4) implies statement (2). 

It remains to show that (1) is equivalent to (5). Assume first that $(A,\diamond,\circ)$ is a truss with cocycle $\sigma$. Then, using the truss distributive law and Lemma~\ref{lem.diamond.inv}, we can compute, for all $a,b,c,d\in A$,
 \begin{align*}
a\circ [b,c,d] &= a\circ \left(b\diamond c^\diamond \diamond d\right)
= (a\circ b)\diamond \sigma(a)^\diamond \diamond \left(a\circ \left(c^\diamond \diamond d\right)\right)\\
&= (a\circ b)\diamond \sigma(a)^\diamond \diamond \sigma(a) \diamond (a\circ c)^\diamond \diamond (a\circ d) = [a\circ b, a\circ c, a\circ d],
\end{align*}
hence the distributive law \eqref{heap.distribute} holds. 

Conversely, setting $c=1_\diamond$ in \eqref{heap.distribute}  and using \eqref{sigma} we obtain the truss distributive law (for all $a,b,d\in A$). 

This completes the proof of the theorem.
\end{proof}

\begin{rem}\label{rem.heap}
The characterisation of truss distributive law as the ditstributive law between the heap and semigroup operations (i.e.\ the equivalence of statements (1) and (5) in Theorem~\ref{thm.equiv})  was observed by Michael Kinyon \cite{Kin:pri}. 
The equivalence of statements (1) and (5) in Theorem~\ref{thm.equiv} suggests a generalisation of the notion of a skew truss, whereby one considers a system $(A, [-,-,-],\circ)$ such that $(A,[-,-,-])$ is a heap or herd, i.e.\
$[-,-,-]$ is a ternary operation on $A$, such that, for all $a_i\in A$, $i=1,\ldots, 5$,
\begin{equation}\label{heap}
\left[\left[a_1,a_2,a_3\right],a_4,a_5\right] = \left[a_1,a_2,\left[a_3,a_4,a_5\right]\right], \qquad \left[a_1,a_2,a_2\right] = a_1 =\left[a_2,a_2,a_1\right],
\end{equation}
(see \cite[page 170]{Pru:the},  \cite[page 202, footnote]{Bae:ein}  or \cite[Definition~2]{Cer:ter}), 
 $(A,\circ)$ is a semigroup, and $\circ$ distributes over $[-,-,-]$ through \eqref{heap.distribute}. This possibility of generalisation is, however, somewhat deceptive:
\end{rem}

\begin{lem}\label{lem.heap}
Let $(A, [-,-,-])$ be a heap with additional associative operation $\circ$ such that the distributive law \eqref{heap.distribute} holds. Fix  $e\in A$ and consider the induced group structure on $A$,
\begin{equation}\label{group.heap}
a\diamond_e b = [a,e,b], \qquad 1_{\diamond_e} = e, \qquad a^{\diamond_e} = [e,a,e], \qquad \mbox{for all $a,b\in A$};
\end{equation}
cf.\ \cite{Bae:ein}. Then $(A,\diamond_e,\circ)$ is a skew truss with the cocycle $\sigma_e: A\to A$, $\sigma_e(a) = a\circ e$.
\end{lem}
\begin{proof}
Using \eqref{heap}, \eqref{group.heap} and the definition of $\sigma_e$ we can compute, for all $a,b,c\in A$,
\begin{align*}
(a\circ b)\diamond_e\sigma_e(a)^{\diamond_e}\diamond_e (a\circ c) &=
\left[a\circ b, e,\left[\left[e, a\circ e, e\right], e, a\circ c\right]\right]\\
&= \left[a\circ b, e,\left[e, a\circ e, \left[ e, e, a\circ c\right]\right]\right]\\
& =  \left[a\circ b, e,\left[e, a\circ e, a\circ c\right]\right]\\
& =  \left[\left[a\circ b, e, e\right], a\circ e, a\circ c\right]\\
& = \left[a\circ b, a\circ e, a\circ c\right] = a\circ [b,e,c] = a\circ \left(b\diamond_e c\right),
\end{align*}
where the penultimate equality is a consequence of \eqref{heap.distribute}.
\end{proof}

Although the heap point of view on trusses might not lead to more general structures, it can point out to a way of constructing examples of skew trusses, which could be relevant to applications (see Remark~\ref{rem.YB}). For example, it quickly indicates that, given one skew truss, there is a whole family of skew trusses:

\begin{cor}\label{cor.family}
Let $(A,\diamond,\circ,\sigma)$ be a skew truss. Fix an element $e\in A$ and define new group operation on $A$ by
$$
a\diamond_e b := a\diamond e^{\diamond}\diamond b.
$$
Then $(A,\diamond_e,\circ)$ is a skew truss with the cocycle $\sigma_e: A\to A$, $\sigma_e(a) = a\circ e$.
\end{cor}
\begin{proof}
If $[-,-,-]$ is the heap operation associated to $\diamond$ as in Theorem~\ref{thm.equiv}~(5), then $a\diamond_e b = [a,e,b]$. Hence the assertion follows by Lemma~\ref{lem.heap}.
\end{proof}

In particular, if $e=\sigma^{n-1}(1_\diamond)$ in Corollary~\ref{cor.family}, then the corresponding skew truss has cocycle $\sigma^n$, since $a\circ e = \sigma^{n-1}(a\circ 1_\diamond) = \sigma^{n}(a)$, by Lemma~\ref{lem.coc.prop}.

To every skew truss $(A,\diamond,\circ,\sigma)$ one can associate two actions of the semigroup $(A,\circ)$ on $A$ by group endomorphisms of $(A,\diamond)$. With respect to these actions the map $\sigma$ satisifies  conditions reminiscent of one-cocycle conditions for group actions; this motivates the terminology used for $\sigma$. 

\begin{thm}\label{thm.brace}
Let  $(A,\diamond,\circ,\sigma)$ be a  skew truss. 
\begin{zlist}
\item 
The semigroup $(A,\circ)$ acts from the left on $(A,\diamond)$ by group endomorphisms through  $\lambda, \mu: A\to \mathrm{Map}(A,A)$, defined by
\begin{subequations}\label{action.maps}
\begin{equation}\label{lam.sig} 
\lambda: a\mapsto \lambda_a, \qquad \lambda_a(b) = \sigma(a)^\diamond \diamond (a\circ b), 
\end{equation}
\begin{equation}\label{mu}
\mu: a\mapsto \mu_a, \qquad \mu_a(b)= (a\circ b)\diamond\sigma(a)^\diamond.
\end{equation}
\end{subequations}
\item The map 
$\sigma$ satisfies the following {\em cocycle conditions},  for all $a,b\in A$,
\begin{equation}\label{sig.cocycle}
\sigma(a\circ b) = \sigma(a)\diamond (a\la \sigma(b)), \qquad \sigma(a\circ b) = (a\bla \sigma(b)) \diamond \sigma(a), 
\end{equation}
where
$a\la b = \lambda_a(b)$ and $a\bla b = \mu_a(b)$ are actions implemented by $\lambda$ and $\mu$.
 \end{zlist} 
\end{thm}
\begin{proof}
The theorem is proven by direct calculations that use the associativity of both operations and the truss distributive law \eqref{semi.brace}. Explicitly, for all $a,b,c\in A$,
\begin{align*}
\lambda_a(b\diamond c) &=\sigma(a)^\diamond  \diamond \left(a\circ\left(b\diamond c\right)\right)= \sigma(a)^\diamond  \diamond (a\circ b) \diamond \sigma(a)^\diamond\diamond(a\circ c)
= \lambda_a(b)\diamond \lambda_a(c).
\end{align*}
Similarly,
\begin{align*}
\mu_a(b\diamond c) &= \left(a\circ\left(b\diamond c\right)\right)\diamond\sigma(a)^\diamond = (a\circ b) \diamond \sigma(a)^\diamond\diamond(a\circ c)\diamond\sigma(a)^\diamond 
= \mu_a(b)\diamond \mu_a(c),
\end{align*}
which proves that, for all $a$, the maps $\lambda_a$ and $\mu_a$ are group endomorphisms. 

Furthermore, for all $a,b,c\in A$,
\begin{align*}
\lambda_a \left(\lambda_b(c)\right) &= \sigma(a)^\diamond \diamond \left(a\circ \lambda_b(c)\right)   
= \sigma(a)^\diamond \diamond \left(a\circ \left(\sigma(b)^\diamond \diamond\left(b\circ c\right) \right) \right) 
\\
&= \sigma(a)^\diamond \diamond \sigma(a)\diamond \left(a\circ \sigma(b)\right)^\diamond \diamond\left(a\circ b\circ c\right)\\
& = \sigma(a\circ b)^\diamond \diamond \left(a\circ b\circ c\right) = \lambda_{a\circ b}(c),
\end{align*}
where the third equality follows by \eqref{binv.c} in Lemma~\ref{lem.diamond.inv} and the penultimate equality follows by \eqref{equivariant}. Therefore, $\lambda$ defines an action as claimed. In a similar way, 
\begin{align*}
\mu_a \left(\mu_b(c)\right) &= \left(a\circ \mu_b(c)\right) \diamond\sigma(a)^\diamond  = \left(a\circ\left(\left(b\circ c\right)\diamond\sigma(b)^\diamond \right)\right)\diamond\sigma(a)^\diamond \\
&= \left(a\circ b\circ c\right)\diamond \left(a\circ \sigma(b)\right)^\diamond \diamond\sigma(a) \diamond\sigma(a)^\diamond\\
& = \left(a\circ b\circ c\right)\diamond\sigma(a\circ b)^\diamond = \mu_{a\circ b}(c),
\end{align*}
where the third equality follows by \eqref{b.cinv} in Lemma~\ref{lem.diamond.inv} and the penultimate equality follows by \eqref{equivariant}. Therefore, $\mu$ also defines an action of $(A,\circ)$ on $A$. This proves part (1) of the theorem.

Using \eqref{equivariant} one quickly finds
$$
\sigma (a\circ b) = a\circ\sigma(b) = \sigma(a) \diamond \sigma(a)^\diamond \diamond \left(a\circ\sigma(b)\right)= \sigma(a) \diamond  \left(a\la \sigma(b)\right),
$$
and 
$$
\sigma (a\circ b) = a\circ\sigma(b) = \left(a\circ\sigma(b)\right)\diamond\sigma(a)^\diamond \diamond\sigma(a) = \left(a\bla \sigma(b)\right)\diamond \sigma(a),
$$
which proves \eqref{sig.cocycle} and thus completes the proof of the theorem.
\end{proof}

Note that  the actions associated to the truss in Example~\ref{ex.truss} mutually coincide  and are equal to $a\la b = a\bla b = 1_\diamond$. In the case of the left skew ring, actions are mutually equal and given by $\circ$. Finally, in the case of a skew brace, the actions come out as
$$
a\la b = a^\diamond \diamond (a\circ b), \qquad a\bla b = (a\circ b)\diamond a^\diamond .
$$

In general we note that the actions $\la$, $\bla$ are related by
$$
\sigma(a)\diamond (a\la b) = (a\bla b)\diamond \sigma(a), \qquad \mbox{for all $a,b \in A$},
$$
and that the actions mutually coincide if the group $(A,\diamond)$ is abelian. In view of Theorem~\ref{thm.equiv} and Theorem~\ref{thm.brace}, and using the notation introduced in the latter,  the skew truss distributive law can be presented in (at least) three equivalent ways,
\begin{align*}
a\circ (b\diamond c) = (a\bla b)\diamond (a\circ c)
= (a\circ b)\diamond\sigma(a)^\diamond\diamond (a\circ c)
= (a\circ b)\diamond (a\la c),
\end{align*}
i.e.\ there are one-to-one correspondences between $\sigma$, $\lambda$ and $\mu$ in Theorem~\ref{thm.equiv}.

\begin{prop}\label{prop.morph}
Let $(A,\diamond_A,\circ_A,\sigma_A)$ and  $(B,\diamond_B,\circ_B,\sigma_B)$ be skew trusses. A function $f:A\to B$  that is a morphism of corresponding groups and semigroups 
renders commutative the following diagrams:
\begin{equation}\label{mor.diag.sig}
\xymatrix{ A\ar[rr]^-{f}\ar[d]^{\sigma_A} && B\ar[d]_{\sigma_B} \\
A\ar[rr]^-{f} && B},
\end{equation}
\begin{equation}\label{mor.diag.mu.lam}
\xymatrix{A \ar[r]^-{\lambda^A}\ar[d]^{f} & \mathrm{Map}(A,\!A) \ar[r]^{f_*} & \mathrm{Map}(A,\!B)\\
B \ar[rr]^-{\lambda^B} && \mathrm{Map}(B,\!B) \ar[u]_{f^*},}\quad 
\xymatrix{A \ar[r]^-{\mu^A}\ar[d]^{f} & \mathrm{Map}(A,\!A) \ar[r]^{f_*} & \mathrm{Map}(A,\! B)\\
B \ar[rr]^-{\mu^B} && \mathrm{Map}(B,\!B) \ar[u]_{f^*},}
\end{equation}
where
$\lambda^{A,B}$ and $\mu^{A,B}$ are the actions corresponding to $\sigma_{A,B}$ through equations \eqref{action.maps} in Theorem~\ref{thm.brace}.
\end{prop}
\begin{proof}
Assume that $f:A\to B$ is a morphism of corresponding groups and semigroups. Using \eqref{sigma}, we find, for all $a\in A$,
$$
\sigma_B(f(a)) = f(a)\circ_B 1_{\diamond_B} = f(a)\circ_B f(1_{\diamond_A}) = f( a\circ_A 1_{\diamond_A}) = f(\sigma_A(a)).
$$ 
Hence the diagram \eqref{mor.diag.sig} is commutative, as required.

The commutativity of the first of diagrams \eqref{mor.diag.mu.lam} is equivalent to
\begin{equation}\label{mor.eq}
 \lambda^B_{f(a)}\left(f\left(a'\right)\right) = f\left(\lambda^A_a\left(a'\right)\right),
\end{equation}
 for all $a,a'\in A$. Using the relations between actions and cocycles \eqref{lam.sig} as well as the facts that $f$ is a morphism of semigroups (on the left hand side of the above equality) and a morphism of groups (on the right hand side), we obtain that  condition \eqref{mor.eq} is equivalent to 
\begin{equation}\label{mor.eq'}
\sigma_{B}\left(f\left(a\right)\right)^\diamond\diamond f\left(a\circ a'\right) = f\left(\sigma_A\left(a\right)\right)^\diamond\diamond f\left(a\circ a'\right).
\end{equation}
By the cancellation laws, \eqref{mor.eq'} is equivalent to
\begin{equation}\label{mor.eq.sig}
\sigma_{B}\left(f\left(a\right)\right) = f\left(\sigma_A\left(a\right)\right),
\end{equation}
for all $a\in A$, which holds by the commutativity of \eqref{mor.diag.sig}.

Evaluated on elements and in terms of the actions \,$\bla$\,, the second of diagrams \eqref{mor.diag.mu.lam} reads
\begin{equation}\label{mor.eq.mu}
f\left(a\bla b\right) = f(a)\bla f(b).
\end{equation}
Provided that $f$ is both a group and semigroup homomorphism, the equivalence of \eqref{mor.eq.mu} with the formula \eqref{mor.eq.sig} follows from the definition \,$\bla$\, by the analogous arguments as in the case of the other action, and then assertion follows by the commutativity of \eqref{mor.diag.sig}.
 \end{proof}
 
 In view of Proposition~\ref{prop.morph} a map that preserves group and semigroup operations also preserves all the other truss structure (cocycles and actions). Thus, as is the case for skew braces and skew rings or near-rings, by a {\em morphism of skew trusses} we mean a map which is a morphism of corresponding groups and semigroups. 
 
 \begin{rem}\label{rem.heap.morph}
 The interpretation of the truss distributive law as a distributive law between a ternary heap and a binary semigroup operations in Theorem~\ref{thm.equiv} suggests another possible definition of a morphism of skew trusses;
 in this context a map $f:A\to B$ that is a morphism of corresponding heaps, i.e.\ for all $a_1,a_2,a_3\in A$,
 $$
 f\left(\left[a_1,a_2,a_3\right]\right) =  \left[f(a_1),f(a_2),f(a_3)\right],
 $$
 and a morphism of semigroups would serve as a morphism of skew trusses.\footnote{This suggestion was made by Michael Kinyon \cite{Kin:pri}.}
 \endeg
 \end{rem}

A truss structure can be ported to any group or a semigroup through an isomorphism. 
 
 \begin{lem}\label{lem.ind}
 Let $(A,\diamond,\circ,\sigma)$ be a skew  truss.
 \begin{zlist}
 \item  Let $(B,\ast)$ be a group. Given an isomorphim of groups $f:B\to A$, define a binary operation on $B$ by porting $\circ$ to $B$ through $f$, i.e.\
 $$
 a\circ_f b = f^{-1}\left(f(a)\circ f(b)\right), \qquad \mbox{for all $a,b\in B$}.
 $$
 Then $(B,\ast,\circ_f)$ is a skew truss isomorphic to $(A,\diamond,\circ,\sigma)$, with the cocycle $\sigma_B := f^{-1}\sigma f$.
 \item   Let $(B,\bullet)$ be a semigroup. Given an isomorphim of semigroups $g:B\to A$, define a binary operation on $B$ by porting $\diamond$ to $B$ through $g$, i.e.\
 $$
 a\diamond_g b = g^{-1}\left(g(a)\diamond g(b)\right), \qquad \mbox{for all $a,b\in B$}.
 $$
 Then $(B,\diamond_g,\bullet)$ is a skew truss isomorphic to $(A,\diamond,\circ,\sigma)$, with the cocycle $\sigma_B := g^{-1}\sigma g$.
 \end{zlist}
 \end{lem}
 \begin{proof}
 We only prove statement (1). Statement (2) is proven in a similar way (see the proof of Lemma~\ref{lem.ind.Hopf} for additional hints).
 
 Since $(A,\circ)$ is a semigroup, so is $(B,\circ_f)$ and $f$ is an isomorphism of semigroups. Furthermore, for all $a,b,c\in B$,
 \begin{align*}
 a\circ_f(b\ast c) &= f^{-1}\left(f(a)\circ f(b\ast c)\right) = f^{-1}\left(f(a)\circ \left(f(b)\diamond f(c)\right)\right)\\
 &= f^{-1}\left(\left(f(a)\circ f(b)\right)\diamond \sigma(f(a))^\diamond \diamond \left(f(a)\circ f(c)\right)\right)\\
 &= f^{-1}\left(\left(f(a)\circ f(b)\right)\right)\ast f^{-1}\left(\sigma(f(a))\right)^\ast \ast f^{-1} \left(f(a)\circ f(c)\right)\\
 &= (a\circ_f b)\ast \sigma_B(a)^\ast \ast (a\circ_f c),
 \end{align*}
 where the second and the fourth equalities follow by the fact that $f$ is a group isomorphism, while the third equality is a consequence of \eqref{semi.brace}. Hence   $(B,\ast,\circ_f)$ is a skew truss with the cocycle $\sigma_B = f^{-1}\sigma f$, as stated. 
 \end{proof}

\begin{rem}\label{rem.hierarchy}
Although the group $(A,\diamond_e)$ in Corollary~\ref{cor.family} is isomorphic to $(A,\diamond)$ via the map
$$
f : (A,\diamond) \to (A,\diamond_e), \qquad a\mapsto a\diamond e,
$$
the map $f$ is not an isomorphism of skew trusses $(A,\diamond, \circ,\sigma)$ and $(A,\diamond_e,\circ,\sigma_e)$ in general. On the other hand, following the porting procedure described in Lemma~\ref{lem.ind}, one can construct a skew truss isomorphic to $(A,\diamond_e, \circ,\sigma_e)$ in which the group structure is given by $\diamond$ and the semigroup structure is deformed:
\begin{align*}
a\circ_e b &= f^{-1}\left(f(a)\circ f(b)\right) = \left(\left(a\diamond e\right)\circ\left(b\diamond e\right)\right)\diamond e^\diamond\\
&= \left(\left(a\diamond e\right)\circ b\right)\diamond \sigma(a\diamond e)^\diamond\diamond \sigma_e(a\diamond e)\diamond e^\diamond,
\end{align*}
where the second expression follows by \eqref{semi.brace} and Corollary~\ref{cor.family}. In view of Lemma~\ref{lem.ind}, the cocycle corresponding to $(A,\diamond,\circ_e)$ comes out as
$$
\tau_e : a\mapsto \sigma_e(a\diamond e)\diamond e^\diamond.
$$
The map $f$ is an isomorphism from $(A,\diamond,\circ_e,\tau_e)$ to $(A,\diamond_e,\circ,\sigma_e)$.
\endeg
\end{rem}

By Lemma~\ref{lem.coc.prop}, a cocycle $\sigma$ is left equivariant with respect to the operation $\circ$; see \eqref{equivariant}. If $1_\diamond$  is central in $(A,\circ)$, then $\sigma$ is almost multiplicative with respect to $\diamond$:
\begin{lem}\label{lem.sig.add}
Let $(A,\diamond,\circ,\sigma)$ be a skew truss, such that $1_\diamond$ is central in $(A,\circ)$.
Then, for all positive integers $n$, and for all $a,b\in A$,
\begin{equation}\label{add}
\sigma^n(a\diamond b)  = \sigma^n(a)\diamond \sigma^n(1_\diamond)^{\diamond}\diamond \sigma^n(b).
\end{equation}
If, in addition also the $n$-th  power of $1_\diamond$ in $(A,\circ)$ is central, then
\begin{equation}\label{add.act}
\sigma^n (a\diamond b)  =\sigma^n(a)\diamond (\sigma^{n-1}(1_\diamond) \la b) =  (\sigma^{n-1}(1_\diamond) \bla a)\diamond\sigma^n(b),
\end{equation}
where $\la$ and $\bla$ are the left actions of $(A,\circ)$ on $(A,\diamond)$  defined in Theorem~\ref{thm.brace}. 
\end{lem}
\begin{proof}
The lemma is proven inductively. We use the definition of $\sigma$ \eqref{sigma}, the centrality of $1_\diamond$ in $(A,\circ)$ and the compatibility condition \eqref{semi.brace}  to obtain
\begin{eqnarray*}
\sigma(a\diamond b) \!\!\!&=&\!\!\!  (a\diamond b)\circ 1_\diamond = 1_\diamond \circ (a\diamond b) \\
  \!\!\!&=&\!\!\! (1_\diamond \circ a)\diamond \sigma(1_\diamond)^\diamond \diamond (1_\diamond \circ b)
=\sigma(a)\diamond \sigma(1_\diamond)^{\diamond}\diamond \sigma(b).
  \end{eqnarray*} 
Setting $b=a^{\diamond}$ in the just derived equality we obtain
\begin{equation}\label{sig.adiam}
\sigma(a^\diamond) = \sigma(1_\diamond)\diamond\sigma(a)^\diamond\diamond\sigma(1_\diamond).
\end{equation}
Assume that \eqref{add} holds for some $k\in \NN$. Then, 
\begin{eqnarray*}
\sigma^{k+1}(a\diamond b) \!\!\!&=&\!\!\!  \sigma\left(\sigma^k(a)\diamond \sigma^k(1_\diamond)^{\diamond}\diamond \sigma^k(b)\right)\\
 \!\!\!&=&\!\!\!\sigma^{k+1}(a)\diamond \sigma(1_\diamond)^\diamond \diamond\sigma\left(\sigma^k(1_\diamond)^{\diamond}\right)\diamond \sigma(1_\diamond)^\diamond\diamond \sigma^{k+1}(b)\\
 \!\!\!&=&\!\!\!\ \sigma^{k+1}(a)\diamond \sigma^{k+1}(1_\diamond)^{\diamond}\diamond \sigma^{k+1}(b),
 \end{eqnarray*}
 by the inductive assumption and \eqref{sig.adiam}. The first assertion then follows by the principle of mathematical induction.
 
 In view of \eqref{sigma}, the $n$-th power of $1_\diamond$ in $(A,\circ)$ is equal to $\sigma^{n-1}(1_\diamond)$. Using its centrality as well as \eqref{sigma} and \eqref{equivariant}, we can compute, for all $b\in A$,
 \begin{align*}
 \sigma^n(1_\diamond)^\diamond\diamond \sigma^n(b) & = \sigma\left( \sigma^{n-1}(1_\diamond)\right)^\diamond\diamond \left(b\circ \sigma^{n-1}(1_\diamond) \right)\\
 &= \sigma\left( \sigma^{n-1}(1_\diamond)\right)^\diamond\diamond \left(\sigma^{n-1}(1_\diamond) \circ b\right) =
 \sigma^{n-1}(1_\diamond) \la b.
 \end{align*}
 The formula \eqref{add.act} follows now by \eqref{add}. The second equality in \eqref{add.act} is proven by similar arguments.
\end{proof}

\begin{cor}\label{cor.group.hom}
Let $(A,\diamond,\circ,\sigma)$ be a skew truss, such that $1_\diamond$ is central in $(A,\circ)$.
For any $n\in \NN$, let $\diamond_n$ be the group operation modified by $e=\sigma^n(1_\diamond)$ as in Corollary~\ref{cor.family}. Then $\sigma^n$ is a group homomorphism from $(A,\diamond)$ to $(A,\diamond_{n})$. 
\end{cor}
\begin{proof}
The assertion follows immediately from the formula \eqref{add} in Lemma~\ref{lem.sig.add} and the definition of $\diamond_n$ through Corollary~\ref{cor.family}, and by the comments following the latter.
\end{proof} 

We end this section by noting similar interpretation of $\sigma$ from the heap point of view.

\begin{lem}\label{lem.heap.mor}
Let $(A,\diamond,\circ,\sigma)$ be a skew truss, such that $1_\diamond$ is central in $(A,\circ)$, and let $[-,-,-]$ be the heap ternary operation associated to $(A,\diamond)$ as in Theorem~\ref{thm.equiv}(5). Then  $\sigma$ is an endomorphism of the heap $(A,[-,-,-])$.
\end{lem}
\begin{proof}
In view of \eqref{sigma}, centrality of  $1_\diamond$  in $(A,\circ)$ and the distributive law \eqref{heap.distribute}, one computes, for all $a,b,c \in A$,
\begin{align*}
\sigma \left(\left[a,b,c\right]\right) &= \left[a,b,c\right] \circ 1_\diamond = 1_\diamond\circ \left[a,b,c\right]
= \left[ 1_\diamond\circ a,\;  1_\diamond\circ b,\;  1_\diamond\circ c\right]\\
&= \left[a\circ 1_\diamond,\; b\circ 1_\diamond,\; c\circ 1_\diamond\right] = \left[\sigma(a), \sigma(b), \sigma(c)\right],
\end{align*}
i.e.\ $\sigma$ is an endomorphism of the heap $(A,[-,-,-])$; see Remark~\ref{rem.heap.morph}.
\end{proof}

\section{Piths of truss morphisms}\label{sec.kernel}
To any morphism of skew trusses one can associate a set which contains the kernel of the morphism as a group homomorphism and  which is closed under the semigroup operation $\circ$.
\begin{defn}\label{def.pith}
Let  $f$ be a morphism from a skew truss $(A,\diamond, \circ,\sigma_A)$ to a skew truss $(B,\diamond, \circ,\sigma_B)$. For any $n\in \NN$, we define a subset of $A$,
\begin{equation}\label{kernel.slice}
\left(\pi f\right)^n := \{ a\in A\; |\;  f\left(a\right) = \sigma_B^{n}(1_\diamond)\}
\end{equation}
and call it the {\em $n$-th chamber of the pith of $f$}. The {\em pith} of $f$ is defined by
\begin{equation}\label{kernel}
\pi f := \bigcup_{n\in \NN} \left(\pi f\right)^n = \{a\in A\; |\; \exists m\in \NN,\; f\left(a\right) = \sigma_B^{m}(1_\diamond)\}.
\end{equation}
\end{defn}

Note that none of the chambers of the pith $\pi f$ is empty, since in view of the commutative diagram \eqref{mor.diag.sig} in Proposition~\ref{prop.morph}, $\sigma_A^n(1_\diamond) \in \left(\pi f\right)^n$, for all $n\in \NN$.

\begin{ex}\label{ex.ker}
Let $(A,\diamond, \circ, \sigma_A)$ be a skew truss and let $(B,\diamond, \circ, \sigma_B)$ be a skew truss associated to an   idempotent endomap $\sigma_B\in \mathrm{Map}(B,B)$  as in Example~\ref{ex.truss}. Let $f$ be a morphism from  $(A,\diamond, \circ, \sigma_A)$ to $(B,\diamond, \circ, \sigma_B)$.
Since 
$$
\sigma_B^n = \begin{cases} \sigma_B & \mbox{if $n>0$}\cr
\mathrm{id}_B & \mbox{if $n=0$},
\end{cases} 
$$
$(\pi f)^n =  (\pi f)^1$, for all $n>0$, and so the pith of $f$ comes out as: 
$$
 \pi f  = \{a\in A\; |\; f(a) = 1_\diamond\} \cup \{a\in A\; |\; f(a) =\sigma_B(1_\diamond)\}. 
$$
\endeg
\end{ex}

Obviously, the zeroth chamber of the pith of $f$  is equal to the kernel of $f$ as a group homomorphism, i.e.\ $\ker f = (\pi f)^0$. Both in the case of skew rings, i.e.\ when $\sigma(a) = 1_\diamond$, and  that of skew braces, in which case $\sigma(a) = a$, all chambers of the pith of a morphism coincide with the zeroth chamber. Thus, in particular, the pith of  a morphism of skew braces is equal to its kernel as defined in \cite[Definition~1.2]{GuaVen:ske}.  

Since $(\pi f)^0$ is the kernel of a group morphism it is a normal subgroup of $(A,\diamond)$. The structure of $\pi f$ with respect to the semigroup operation $\circ$ is investigated in the following
\begin{lem}\label{lem.kernel}
Let  $f$ be a morphism from a skew truss $(A,\diamond, \circ,\sigma_A)$ to a skew truss $(B,\diamond, \circ,\sigma_B)$. Then:
\begin{zlist}
\item For all $a\in \left(\pi f\right)^m$ and $b\in \left(\pi f\right)^n$, $a\circ b \in \left(\pi f\right)^{m+n+1}$.
\item The pith of $f$ is a sub-semigroup of $(A,\circ)$.
\item For all $a\in \left(\pi f\right)^m$, $\sigma_A(a) \in \left(\pi f\right)^{m+1}$.
\item If, for any $(m,n)\in \NN^2$, $\sigma_B^m(1_\diamond) = \sigma_B^n(1_\diamond)$ implies that $m= n$, then the chambers of the pith of $f$ define a partition of $\pi f$, and $(\pi f,\circ)$ is an $\NN_+$-graded semigroup with grading given by the chamber labels shifted by one.
\end{zlist}
\end{lem}
\begin{proof}
(1) Let us take any $a$ and $b$ for which there exist $m,n\in \NN$ such that
$$
f\left(a\right) = \sigma_B^{m}(1_\diamond) \quad \mbox{and} \quad f\left(b\right)= \sigma_B^{n}(1_\diamond).
$$
Then
\begin{align*}
f\left(a\circ b\right) &=f\left(a\right)\circ f\left(b\right)
= f(a)\circ \sigma_B^{n}\left(1_\diamond\right)\\
&= \sigma_B^{n}\left(f(a)\circ 1_\diamond\right)
 =  \sigma_B^{n}\left(\sigma^m_B(1_\diamond)\circ 1_\diamond\right) 
=  \sigma_B^{m+n+1}(1_\diamond),
\end{align*}
where we have used that $f$ is a semigroup morphism and   the properties of cocycles listed in Lemma~\ref{lem.coc.prop}. Therefore, $a\circ b\in  \left(\pi f\right)^{m+n+1}$.

(2) This follows immediately from part (1).

(3) This follows from the commutative diagram \eqref{mor.diag.sig} in Proposition~\ref{prop.morph} and  the definition of pith chambers in \eqref{kernel.slice}.

(4) An element $a$ of the pith of $f$ belongs to two different chambers if and only if there exist $m\neq n$ such that
$$
f\left(a\right) = \sigma_B^{m}(1_\diamond) \quad \mbox{and} \quad   f\left(a\right)= \sigma_B^{n}(1_\diamond).
$$
Hence the assumption means that 
any two chambers labeled by different integers have empty intersections, i.e.\ the chambers give a partition of $\pi f$. This defines the grading function
$$
\deg : \pi f \to \NN_+, \qquad \deg(a) =n+1\quad \mbox{if and only if} \quad a\in \left(\pi f\right)^{n},
$$
which is compatible with the semigroup operation by part (1). Therefore, $\pi f$ is an $\NN_+$-graded semigroup as stated.
 \end{proof}
 
 \begin{rem}\label{rem.pith.grad}
 It is clear from Lemma~\ref{lem.kernel}~(4) that, without any assumptions on $\sigma_B$, the disjoint union of pith chambers,
 $$
 \mathrm{Pi} f := \bigsqcup_{n\in \NN} \left(\pi f\right)^n,
 $$
 is an $\NN_+$-graded semigroup. 
 \endeg
 \end{rem}

\section{Skew trusses and skew braces}\label{sec.truss.brace}
The aim of this section is to determine when a cocycle in a skew truss is invertible and then to show that if the semigroup $(A,\circ)$ is a group, then the cocycle is invertible and hence there is an associated skew brace. In other words, one can identify skew trusses built on two groups with skew braces (this identification is not through an isomorphism in the category of trusses, though).

\begin{lem}\label{lem.sigma.inverse}
Let  $(A,\diamond,\circ,\sigma)$ be a  skew truss. Then the following statements are equivalent:
\begin{rlist}
\item The cocycle $\sigma$ is bijective.
\item The operation $\circ$ has a right identity with respect to which $1_\diamond$ is invertible.
\end{rlist}
\end{lem}
\begin{proof}
Assume first that $\sigma$ is a bijection and set
\begin{equation}\label{e,u}
e = \sigma^{-1}(1_\diamond),\qquad u = \sigma^{-1}(e).
\end{equation} 
In consequence of the equivariance property \eqref{equivariant}, $\sigma^{-1}$ is also left equivariant, i.e., for all $a,b\in A$,
\begin{equation}\label{equiv.sig-1}
\sigma^{-1}(a\circ b) = a\circ  \sigma^{-1}( b).
\end{equation}
Therefore, for  all $a\in A$,
$$
a = \sigma^{-1}(\sigma (a)) = \sigma^{-1}(a\circ 1_\diamond) = a\circ e,
$$
where the second equality follows by \eqref{sigma}, and the last one by \eqref{equiv.sig-1} and \eqref{e,u}. Hence $e$ is a right identity for $\circ$. Furthermore, using the above calculation as well as \eqref{equiv.sig-1} and the definition of $u$  in \eqref{e,u}  one finds, for all $a\in A$,
$$
\sigma^{-1} (a) =\sigma^{-1}(a\circ e) = a\circ \sigma^{-1}(e) = a\circ u.
$$
Consequently,
$$
1_\diamond \circ u = \sigma^{-1}(1_\diamond) = e = \sigma(u) = u\circ 1_\diamond.
$$
Therefore $u$ is the inverse of $1_\diamond$ with respect to the right identity $e$.

In the converse direction, denote by $u$ the inverse of $1_\diamond$ with respect to a right identity for $\circ$. Then the inverse of $\sigma$ is given by $\sigma^{-1}(a) = a\circ u.$
 \end{proof}
 \begin{cor}\label{cor.sigma.inverse}
 If $(A,\diamond,\circ,\sigma)$ is a  skew truss such that
 $(A,\circ)$ is a monoid with identity $1_\circ$, then 
\begin{equation}\label{sig.uni}
\sigma(1_\circ) = 1_\diamond.
\end{equation}
 \end{cor}
\begin{proof}
The lemma follows directly from  \eqref{sigma} in Lemma~\ref{lem.coc.prop}. 
\end{proof}

\begin{lem}\label{lem.brace}
If  $(A,\diamond,\circ,\sigma)$ is a  skew truss such that
$(A,\circ)$ is a group, then $A$ is a skew brace with respect to $\diamond$ and operation $\bullet$ defined by
\begin{equation}\label{bullet}
a\bullet b = \sigma^{-1}(a)\circ b = a\circ 1_\diamond^\circ\circ b, \qquad \mbox{for all $a,b\in A$}.
\end{equation}
\end{lem}
\begin{proof}
Since all elements of $A$ are units in $(A,\circ)$, so is the identity $1_\diamond$ and the map $\sigma$ is bijective by Lemma~\ref{lem.sigma.inverse}. Hence, by Theorem~\ref{thm.brace}, $\sigma$ is a bijective one-cocycle from $(A,\circ)$ to $(A,\diamond)$. By \cite[Proposition~1.11]{GuaVen:ske}, $A$ is a skew brace with operations $\diamond$ and the induced operation
$$
a\bullet b = \sigma\left(\sigma^{-1}(a)\circ\sigma^{-1}(b)\right) = \sigma^{-1}(a)\circ b = a\circ 1_\diamond^\circ\circ b,
$$
where the first equality follows by \eqref{equivariant} and the second one by the fact that $\sigma^{-1}(a) = a\circ 1_\diamond^\circ$ (compare the proof of Lemma~\ref{lem.sigma.inverse}, where necessarily $u = 1_\diamond^\circ$ in this case). 
\end{proof}

\begin{rem}\label{rem.brace}
Note that $(A,\diamond,\circ,\sigma)$ is not isomorphic with $(A,\diamond, \bullet,\id)$ in the category of trusses. Indeed, no skew brace can be isomorphic to a skew truss with a nontrivial cocycle. Such an isomorphism would need to be a bijective map $f$ satisfying $f = \sigma f$, thus yielding the contradictory equality $\sigma = \id$.

The construction in Lemma~\ref{lem.brace} is categorical in the sense that it gives a pair of functors between the category $\mathcal{T}$ of trusses with bijective cocycles and the category of braces $\mathcal{B}$. More precisely, the functors are identities on morphisms and on objects are given by
$$
F: \Tt\to \Bb, \quad (A,\diamond,\circ,\sigma)\mapsto (A,\diamond, \bullet),  \qquad G: \Bb\to \Tt \quad (A,\diamond,\bullet)\to (A,\diamond, \bullet,\id).
$$
Obviously, $F\circ G =\id_\Bb$, but these functors are not inverse equivalences, because the existence of a functorial isomorphism $G\circ F\cong \id_\Tt$ would require the existence of isomorphisms between trusses and braces, which is ruled out by the preceding discussion. 
\endeg
\end{rem}

\begin{rem}\label{rem.brace.heap} 
In contrast to the first statement of Remark~\ref{rem.brace}, if $1_\diamond$ is central in $(A,\circ)$, then $\sigma$ is not only an isomorphism of groups $(A,\bullet)$ and $(A,\circ)$, but also an isomorphism of heaps associated to $(A,\diamond)$ by Lemma~\ref{lem.heap.mor}. Should the morphisms in category of trusses be defined as simultanous morphisms of heaps and semigroups, as in Remark~\ref{rem.heap.morph}, then $(A,\diamond,\circ)$ and $(A,\diamond,\bullet)$ would be isomorphic as trusses.
\endeg
\end{rem}

Recall from \cite{Dri:uns} that, for a set $A$, a function $r: A\times A\to A\times A$ is a {\em solution to the set-theoretic Yang-Baxter equation} provided that
$$
(r\times \id)(\id \times r)(r\times \id)=(\id \times r)(r\times \id)(\id \times r) .
$$

\begin{cor}\label{cor.Y-B}
Let  $(A,\diamond,\circ,\sigma)$ be  a  skew truss such that
$(A,\circ)$ is a group. Take any $e\in A$ and define  $r: A\times A\to A\times A$ by
\begin{eqnarray}\label{sol.Y-B}
r(a,b) 
\!\!\!&=&\!\!\! \left(e\diamond a^\diamond\diamond (a\circ e^\circ\circ b), e\circ \left(e\diamond a^\diamond\diamond (a\circ e^\circ\circ b)\right)^\circ\circ a\circ e^\circ \circ b\right). 
\end{eqnarray}
The function $r$ is a bijective solution to the set-theoretic Yang-Baxter equation.
\end{cor}
\begin{proof}
By Corollary~\ref{cor.family}, $(A,\diamond_e,
\circ,\sigma_e)$, where $a\diamond_e b = a\diamond e^\diamond \diamond b$ and $\sigma_e(a) = a\circ e$, is a skew truss. The identity of $(A,\diamond_e)$ is $e$, and since $(A,\circ)$ is a group, $(A,\diamond_e,\bullet_e)$, where $a\bullet_e b = a\circ e^\circ\circ b$,  is a skew brace by Lemma~\ref{lem.brace}. In view of \cite[Theorem~3.1]{GuaVen:ske} there is an associated bijective solution of the Yang-Baxter equation
$$
r_n: A\times A\to A\times A, \quad (a,b)\mapsto \left(a^{\diamond_e}\diamond_e (a\bullet_e b), \left(a^{\diamond_e}\diamond_e (a\bullet_e b)\right)^{\bullet_e} \bullet_e a\bullet_e b\right).
$$
The solution \eqref{sol.Y-B} is simply the translation of  this formula with the help of \eqref{bullet}. 
\end{proof}

\begin{rem}\label{rem.YB}
In view of Lemma~\ref{lem.heap} and Theorem~\ref{thm.equiv}, Corollary~\ref{cor.Y-B} can be restated in terms of heaps. Let $A$ be equipped with a group operation $\circ$ that distributes over a heap operation $[-,-,-]$ on $A$ as in \eqref{heap.distribute}. Then, for all $e\in A$, the function $r:A\times A\to A\times A$,
\begin{equation}\label{YB-heap}
r(a,b) 
= \left([e,a, a\circ e^\circ\circ b], e\circ \left[e,a,a\circ e^\circ\circ b\right]^\circ\circ a\circ e^\circ \circ b\right), 
\end{equation}
is a bijective solution to the set-theoretic Yang-Baxter equation. 

In fact the solution \eqref{YB-heap} can be written as originating from two heap operations. Let 
$$
\langle a, b, c\rangle := a\circ b^\circ \circ c,
$$
be the heap operation associated to $(A,\circ)$. Then \eqref{YB-heap} reads:
\begin{equation}\label{YB-heap.heap}
r(a,b) 
= \left(\left[e,a, \langle a, e, b\rangle\right],\; \langle \langle e , \left[e,a,\langle a, e, b\rangle\right], a\rangle ,  e, b\rangle\right).
\end{equation}
\endeg
\end{rem}

\section{Trusses and rings}\label{sec.ring}
In this section we show that the fact that to any two-sided brace one can associate a ring \cite{Rum:mod} is not brace-specific\footnote{Admittedly, the fact that the ring associated to a brace is a {\em radical} ring is brace-specific.}, but it is a consequence of the truss distributive laws.
\begin{defn} \label{def.truss.abelian}
A set $A$ together with an abelian group operation $+$ and an associative operation $\circ$ is called a {\em two-sided truss}, provided $(A,+,\circ)$ is both a left and right truss with the same cocycle. 
\end{defn}
We will use additive notation for the group $(A,+)$, denoting the neutral element by 0 and the inverse of $a\in A$  by $-a$, and we adopt the convention that $\circ$ takes precedence over $+$. Thus in a two-sided truss, the operations $\circ$ and $+$ are related by
\begin{equation}\label{2.truss}
a\circ(b+c) = a\circ b +a\circ c -\sigma(a), \qquad (a+b)\circ c = a\circ c + b\circ c - \sigma(c),
\end{equation}
and by Lemma~\ref{lem.coc.prop} (and its right-handed version), the cocycle satisfies
\begin{equation}\label{sigma.ab}
\sigma(a) = a\circ 0 = 0\circ a,
\end{equation}
hence $0$ is central in $(A,\circ)$ and we have Lemma~\ref{lem.sig.add} at our disposal, i.e., for all $a,b\in A$, and $n\in \NN$,
\begin{equation}\label{add.ab}
\sigma^n(a+b) = \sigma^n(a) + \sigma^n(b) - \sigma^n(0) = \sigma^n(a) + \sigma^n(b) - \underbrace{0\circ 0\circ\ldots\circ 0}_{n+1-\text{times}}. 
\end{equation}
In this section we show that there is a ring associated to any two-sided truss.

\begin{thm}\label{thm.brace.ring}
Given a  two-sided truss $(A,+,\circ,\sigma)$, define a binary operation $\bullet$ on $A$ by
\begin{equation}
a\bullet b = a\circ b - \sigma(a+b).
\end{equation}
 Then $(A,+,\bullet)$ is a ring. The assignment $(A,+,\circ,\sigma)\mapsto (A,+,\bullet)$ defines a functor (that acts on morphisms as the identity) from the category of trusses to the category of rings.
\end{thm}

\begin{proof}
We first prove that operation $\bullet$ distributes over $+$. Set $e=\sigma(0)$, take any $a,b,c\in A$, and compute
\begin{eqnarray*}
a\bullet(b+c) \!\!\!&=&\!\!\! a\circ (b+c) - \sigma(a+b+c)\\ 
\!\!\!&=&\!\!\! a\circ b  -\sigma(a) + a\circ c - \sigma(a+b+c) \\
 \!\!\!&=&\!\!\! a\circ b + a\circ c -\sigma(a+a+b+c) + e,
 \end{eqnarray*}
 where we used the definition of $\bullet$, then the relation \eqref{sigma.ab}  and finally \eqref{add.ab}. On the other hand
 \begin{eqnarray*}
 a\bullet b+ a\bullet c \!\!\!&=&\!\!\! a\circ b - \sigma(a+b) + a\circ c  - \sigma(a+c)\\
 \!\!\!&=&\!\!\!  a\circ b + a\circ c -\sigma(a+a+b+c) + e,
  \end{eqnarray*}
 where again \eqref{add.ab} has been applied. This proves the left-distributivity. The right-distributivity follows by the left-right symmetry. 
 
In view of the fact that $\sigma(a) = 0\circ a = a\circ 0$, the distributivity of $\bullet$ over addition implies that, for all $a,b\in A$,
\begin{equation}\label{bul.-a}
(-a)\bullet b = a\bullet(-b) = -a\bullet b.
\end{equation}
With this at hand we compute, for all $a,b,c\in A$,
\begin{eqnarray*}
a\bullet(b\bullet c) \!\!\!&=&\!\!\! a\bullet\left(b\circ c-\sigma(b+c)\right) = a\bullet(b\circ c) - a\bullet\sigma(b+c)\\
\!\!\!&=&\!\!\! a\circ b\circ c -\sigma(a+b\circ c)-a\circ\sigma(b+c) +\sigma(a+\sigma(b+c))\\
\!\!\!&=&\!\!\! a\circ b\circ c - b\circ \sigma(c)-\sigma(a)\circ(b+c) +\sigma^2(b+c)\\
\!\!\!&=&\!\!\! a\circ b\circ c - b\circ \sigma(c)-a\circ  \sigma(b) -a\circ  \sigma(c) \\
&&+\sigma^2(a) + \sigma^2(b)+\sigma^2(c) -\sigma(e),
\end{eqnarray*}
where the  first and third equalities follow by the definition  of $\bullet$, the second equality is a consequence of \eqref{bul.-a} and the distributivity of $\bullet$, the fourth one follows by \eqref{add.ab} (with $n=1$) and the two-sided version of \eqref{equivariant}. The final equality is a consequence of \eqref{2.truss} and \eqref{add.ab} (with $n=2$). On the other hand, using the same chain of arguments we obtain
\begin{eqnarray*}
(a\bullet b)\bullet c \!\!\!&=&\!\!\! \left(a\circ b-\sigma(a+b)\right)\bullet c = (a\circ b)\bullet c - \sigma(a+b)\bullet c\\
\!\!\!&=&\!\!\! a\circ b\circ c -\sigma(a\circ b + c)-\sigma(a+b)\circ c +\sigma(\sigma(a+b)+c)\\
\!\!\!&=&\!\!\! a\circ b\circ c - a\circ \sigma(b)-(a+b)\circ \sigma(c)+\sigma^2(a+b)\\
\!\!\!&=&\!\!\! a\circ b\circ c - b\circ \sigma(c)-a\circ  \sigma(b) -a\circ  \sigma(c) \\
&&+\sigma^2(a) + \sigma^2(b)+\sigma^2(c) -\sigma(e).
\end{eqnarray*}
Therefore the operation $\bullet$ satisfies the associative law, and thus $(A,+,\bullet)$ is an associative ring as claimed.

If $f: A\to B$ is a morphism from $(A,+,\circ,\sigma_A)$ to $(B,+,\circ,\sigma_B)$, then for all $a,a'\in A$,
\begin{eqnarray*}
f(a\bullet a') \!\!\!&=&\!\!\! f\left(a\circ a' - \sigma_A(a+a')\right) \\
\!\!\!&=&\!\!\! f(a)\circ f(a') - f\left(\sigma_A(a+a')\right)\\
\!\!\!&=&\!\!\! f(a)\circ f(a') - \sigma_B\left(f\left(a+a'\right)\right)\\
 \!\!\!&=&\!\!\!  f(a)\circ f(a') - \sigma_B\left(f\left(a\right)+f\left(a'\right)\right) = f(a)\bullet f(a'),
\end{eqnarray*}
where the second and the fourth equalities follow by the fact that $f$ is an additive map preserving $\circ$, and the third equality follows by \eqref{mor.diag.sig}. Thus the process of turning a two-sided truss into a ring defines a functor as stated.
\end{proof}

\begin{cor}\label{cor.ring.e}
Given a  two-sided truss $(A,+,\circ,\sigma)$, for any $e$ in the centre of $(A,\circ)$, define  binary operations  $+_e$ and  $\bullet_e$ on $A$ by
\begin{equation}\label{+e}
a+_eb = a+b-e,
\end{equation}
\begin{equation}\label{bullete}
a\bullet_e b = a\circ b - a\circ e - b\circ e + e\circ e +e +e .
\end{equation}
 Then $(A,+_e,\bullet_e)$ is a ring. 
 \end{cor}
 \begin{proof}
 By Corollary~\ref{cor.family} (and its right-handed version), associated to a central element $e$ of $(A,\circ)$ and a two-sided truss $(A,+,\circ,\sigma)$ there is a two-sided truss $(A,+_e,\circ)$, where $+_e$ is given by \eqref{+e}, with cocycle $\sigma_e(a) = a\circ e = e\circ a$ (here the centrality of $e$ ensures that the cycles for the left and right trusses coincide).
  Thus, by Theorem~\ref{thm.brace.ring}, $(A,+_e,\bullet_e)$ is a ring, where
\begin{align*}
 a\bullet_e b &=  a\circ b -_e (a+_e b)\circ e\\
&= a\circ b - \left(a+b-e\right)\circ e + e +e\\
&= a\circ b - \left(a+b\right)\circ e -\sigma\left(e\right) + e\circ e + e +e\\
&=a\circ b - a\circ e - b\circ e + e\circ e +e +e,
\end{align*}
where the second equality follows by the definition of the inverse for the operation $+_e$ in Lemma~\ref{lem.heap} and Corollary~\ref{cor.family},  and by \eqref{+e}. The third and fourth equalities follow by Lemma~\ref{lem.diamond.inv} and the truss distributive law.
 \end{proof}

 \section{Hopf trusses}\label{sec.Hopf}
 In this section we follow \cite{AngGal:Hop} and linearise the notion of a skew truss to obtain its Hopf algebra version.
 
 \begin{defn}\label{def.truss.Hopf}
Let  $(A,\Delta,\eps)$ be a coalgebra over a field $\KK$. Let $\diamond$ be a binary operation on $A$ making $(A,\Delta,\eps)$ a Hopf algebra (with identity denoted by $1_\diamond$ and  the antipode $S$). Let $\circ$ be a binary operation on $A$ making $(A,\Delta,\eps)$ a non-unital bialgebra (i.e.\ $\circ: A\otimes A\to A$ is a morphism of coalgebras satisfying the associative law). We say that $A$ is a {\em left Hopf truss} if there exists a coalgebra endomorphism $\sigma$ of $(A,\Delta,\eps)$, such that
\begin{equation}\label{brace.Hopf}
a\circ(b\diamond c) = \sum (a\sw 1\circ b)\diamond S\sigma(a\sw 2)\diamond (a\sw 3\circ c),
\end{equation}
 for all $a,b,c\in A$.

A {\em right Hopf truss} is defined in a symmetric way. The map $\sigma$ in \eqref{brace.Hopf} is called a {\em cocycle}.
\end{defn}

A Hopf truss is a {\em Hopf brace} in the sense of \cite{AngGal:Hop} provided $(A,\Delta,\circ)$ is a Hopf algebra and the cocycle $\sigma$ is the identity.

As was the case for skew trusses, whatever can be said of a left Hopf truss can be also said (with obvious side variation) about a right Hopf truss, and hence we concentrate on the former and omit the adjective `left'. 

In
a Hopf  truss the cocyle is fully determined by the neutral element of the Hopf algebra and by the operation $\circ$.
\begin{lem}\label{lem.coc.prop.Hopf}
Let $(A,\diamond,\circ,\sigma)$ be a Hopf truss. Then, for all $a\in A$,
\begin{equation}\label{sigma.Hopf}
\sigma(a)= a\circ 1_\diamond.
\end{equation}
Consequently, 
the cocycle $\sigma$ is left $(A,\circ)$-linear, i.e.\ the equality \eqref{equivariant} holds.
\end{lem}
\begin{proof}
Let $\tau(a) = a\circ 1_\diamond$. 
Setting $b=c=1_\diamond$ in the law \eqref{brace.Hopf} we obtain
$$
 a\circ 1_\diamond = a\circ(1_\diamond\diamond1_\diamond) = \sum (a\sw 1\circ 1_\diamond )\diamond S\sigma(a\sw 2) \diamond (a\sw 3\circ 1_\diamond) ,
$$ 
i.e.,
$$
\tau(a) = \sum \tau(a\sw 1)\diamond S\sigma(a\sw 2) \diamond \tau(a\sw 3).
$$
This implies that
\begin{align}\label{long}
\sum \tau(a\sw 1)\diamond & S\tau(a\sw 2) \diamond \sigma(a\sw 3)\nonumber \\ &= \sum
\tau(a\sw 1)\diamond S\sigma(a\sw 2) \diamond  \tau(a\sw 3)\diamond  S\tau(a\sw 4) \diamond  \sigma(a\sw 5).
\end{align}
By the definition $\sigma$ is a coalgebra map, i.e., for all $a\in A$,
\begin{equation}\label{sig.endo}
\sum \sigma(a)\sw 1\ot \sigma(a)\sw 2 = \sum \sigma(a\sw 1)\ot \sigma(a\sw 2), \qquad \eps(\sigma(a)) = \eps(a).
\end{equation}
Since $(A,\circ)$ is a bialgebra and $1_\diamond$ is a grouplike element, also $\tau$ is a coalgebra endomorphism. Using these facts as well as the properties of the antipode we can develop both sides of \eqref{long} to find that $\tau(a) = \sigma(a)$, i.e.\ that \eqref{sigma.Hopf} holds.

The second statement follows by the associativity of $\circ$ and \eqref{sigma.Hopf}.
\end{proof}

\begin{lem}\label{lem.diamond.antipode}
Let $(A,\diamond,\circ,\sigma)$ be a Hopf truss.  Then, for all $a,b,c\in A$,
\begin{align}\label{Sb.c}
a\circ (S(b)\diamond c) &= \sum \sigma(a\sw 1)\diamond S(a\sw 2\circ b)\diamond (a\sw 3\circ c),
\end{align}
and
\begin{align}\label{b.Sc}
a\circ (b\diamond S(c)) &=  \sum (a\sw 1\circ b) \diamond S(a\sw 2\circ c) \diamond\sigma(a\sw 3). 
\end{align}
\end{lem}
\begin{proof}
Let us take any $a,b,c\in A$ and compute
\begin{align*}
\sum \sigma(a\sw 1) &\ot a\sw 2\circ b \ot a\sw 3\circ c = \sum \sigma(a\sw 1)\ot a\sw 2\circ b\sw 1 \ot a\sw 3\circ \left(b\sw 2\diamond S(b\sw 3)\diamond c\right)\\
&= \sum \sigma(a\sw 1)\ot a\sw 2\circ b\sw 1 \ot (a\sw 3\circ b\sw 2)\diamond S\sigma(a\sw 4)\diamond \left (a\sw 5\circ \left(S(b\sw 3)\diamond c\right)\right),
\end{align*}
where the first equality follows by properties of the antipode and the counit and the second one by \eqref{brace.Hopf}. Equation \eqref{Sb.c} follows by applying $\diamond(\diamond\ot \id)(\id\ot S\ot \id)$ to both sides of the equality just derived and by using that $(A,\diamond)$ is a Hopf algebra and $(A,\circ)$ is a bialgebra, both with the same comultiplication and counit, and that $\sigma$ is a coalgebra homomorphism. Equation \eqref{b.Sc} is proven in a similar way.
\end{proof}

Also as in the skew truss case, the distributive law \eqref{brace.Hopf} can be stated in a number of equivalent ways.

\begin{thm}\label{thm.equiv.Hopf}
Let $(A,\Delta,\eps)$ be a coalgebra with two bilinear operations $\diamond$ and $\circ$ making $A$ into a Hopf algebra and a bialgebra respectively. Then the following statements are equivalent:
\begin{zlist}
\item There exists a linear endomorphism $\sigma: A\to A$ such that $(A,\diamond,\circ,\sigma)$ is a Hopf truss.
\item There exists a linear map 
$\lambda : A\ot A \to A$, 
such that, for all $a,b,c\in A$,
\begin{equation}\label{distribute.Hopf}
a\circ (b\diamond c) = \sum (a\sw 1\circ b)\diamond \lambda (a\sw 2\ot c).
\end{equation}
\item There exists a linear map
$
\mu :  A\ot A \to A$, 
such that, for all $a,b,c\in A$,
\begin{equation}\label{distribute.mu.Hopf}
a\circ (b\diamond c) = \sum \mu (a\sw 1\ot b)\diamond (a\sw 2\circ c).
\end{equation}
\item There exist linear maps
$
\kappa,\hat\kappa :  A\ot A \to A$, at least one of which is a coalgebra morphism and
such that, for all $a,b,c\in A$,
\begin{equation}\label{distribute.kap.Hopf}
a\circ (b\diamond c) = \sum \kappa (a\sw 1 \ot b)\diamond \hat\kappa (a\sw 2\ot c).
\end{equation}
\item Define a linear map
$[-,-,-]: A\ot A\ot A\to A$, by $[a,b,c]= a\diamond S(b)\diamond c.$ Then, for all $a,b,c,d\in A$,
\begin{equation}\label{distribute.heap.Hopf}
a\circ \left[b,c,d\right] = \sum \left[a\sw 1\circ b, a\sw 2\circ c, a\sw 3\circ d\right].
\end{equation}
\end{zlist}
\end{thm}
\begin{proof}
(1)$\implies$(2), (3) and (4).  Given $\sigma$, define, for all $a,b\in A$, 
\begin{subequations}
\begin{equation}\label{lam.sig.Hopf}
\lambda (a\ot b) = \sum S\sigma(a\sw 1)\diamond (a\sw 2\circ b),
\end{equation}
\begin{equation}\label{Hopf.act.mu}
 \mu(a\ot b) = \sum (a\sw 1\circ b) \diamond S\sigma(a\sw 2),
\end{equation}
\begin{equation}\label{Hopf.act.kap}
\kappa(a\ot b) = a\circ b, \qquad \hat\kappa(a\ot b) = \sum S\sigma(a\sw 1)\diamond (a\sw 2\circ b).
\end{equation}
\end{subequations}
The coassociativity of $\Delta$ ensures that the satisfaction of \eqref{brace.Hopf} implies the satisfaction of \eqref{distribute.Hopf}, \eqref{distribute.mu.Hopf} and \eqref{distribute.kap.Hopf}. In the case of assertion (4), since the operation $\circ$ is a coalgebra morphism, so is $\kappa$ in \eqref{Hopf.act.kap}.

(2)$\implies$(1) Define a linear endomorphism $\sigma: A\to A$ by $\sigma(a) = a\circ 1_\diamond$. Setting $b = 1_\diamond$ in  \eqref{distribute.Hopf}, we obtain that $a\circ c = \sum \sigma(a\sw 1) \diamond \lambda (a\sw 2\ot c)$. Since $1_\diamond$ is a grouplike element and $\circ$ is a coalgebra map, the map $\sigma$ is a coalgebra map too, i.e.\ it satisfies \eqref{sig.endo}. Using the properties of the antipode and the coalgebra map property of $\sigma$, \eqref{distribute.Hopf} can be developed as
\begin{align*}
a\circ (b\diamond c) &= \sum (a\sw 1\circ b)\diamond S\sigma(a\sw 2)\diamond\sigma(a\sw 3)\diamond  \lambda (a\sw 4\ot c)\\
&= \sum (a\sw 1\circ b)\diamond S\sigma(a\sw 2)\diamond (a\sw 3\circ c),
\end{align*}
as required.
The implication (3)$\implies$(1) is proven in a similar way.

(4)$\implies$(2) or (3) Suppose that $\hat\kappa$ is a coalgebra morphism and define $\tau: A\to A$ by $\tau(a) = \hat\kappa(a\ot 1_\diamond)$.  Setting $c=1_\diamond$ in \eqref{distribute.kap.Hopf} we obtain 
$$
a\circ b =  \sum \kappa (a\sw 1\ot b)\diamond \tau(a\sw 2) .
$$
Since $1_\diamond$ is a grouplike element $\tau$ is a coalgebra endomorphism. With this at hand, \eqref{distribute.kap.Hopf} can be rewritten as
\begin{align*}
a\circ (b\diamond c) &= \sum \kappa (a\sw 1\ot b)\diamond \tau(a\sw 2)\diamond S\tau(a\sw 3) \diamond \hat\kappa (a\sw 4\ot c)\\
& = \sum (a\sw 1\circ b)\diamond S\tau(a\sw 2) \diamond \hat\kappa(a\sw 3\ot c).
\end{align*}
Therefore, \eqref{distribute.kap.Hopf} takes the form of \eqref{distribute.Hopf} (with $\lambda (a\ot c) = \sum S\tau(a\sw 1) \diamond \hat\kappa(a\sw 2\ot c)$), i.e.\ statement (4) implies statement (2). 

If $\kappa$ is a coalgebra morphism, then we define a coalgebra endomorphism of $A$ by  $\tau(a) =\kappa(a\ot 1_\diamond)$ and proceed as above to conclude that (4) implies (3). 

(1)$\implies$(5) Using \eqref{brace.Hopf} and Lemma~\ref{lem.coc.prop.Hopf}, we can compute, for all $a,b,c,d\in A$,
\begin{align*}
a\circ \left[b,c,d\right] &= \sum a\circ \left(b\diamond S(c)\diamond d\right)
= \sum (a\sw 1\circ b)\diamond S\sigma(a\sw 2) \diamond \left(a\sw 3\circ \left(c^\diamond \diamond d\right)\right)\\
&= \sum (a\sw 1\circ b)\diamond S\sigma(a\sw 2) \diamond \sigma(a\sw 3) \diamond S(a\sw 4\circ c)(a\sw 5\circ d)\\
& = \sum \left[a\sw 1\circ b, a\sw 2\circ c, a\sw 3\circ d\right],
\end{align*}
by the fact that $\sigma$ is a coalgebra map and by the antipode axioms. Hence \eqref{distribute.heap.Hopf} holds. 

(5)$\implies$(1) Setting $c=1_\diamond$ in \eqref{distribute.heap.Hopf}, and $\sigma(a) = a\circ 1_\diamond$, we find that \eqref{brace.Hopf} is satisfied (for all $a,b,d\in A$).

This completes the proof of the theorem.
\end{proof}

\begin{thm}\label{thm.truss.Hopf}
Let  $(A,\Delta,\diamond,\circ,\sigma)$ be a Hopf truss.
\begin{zlist}
\item $(A,\diamond)$ is a left Hopf module algebra over $(A,\Delta,\circ)$ with any of the actions
\begin{subequations}
\label{Hopf.act}
\begin{equation}\label{la.Hopf}
\la : A\ot A\to A, \qquad a\la b = \sum S\sigma(a\sw 1)\diamond (a\sw 2\circ b),
\end{equation}
\begin{equation}\label{bla.Hopf}
\bla\:: A\ot A\to A, \qquad   a \bla b = \sum (a\sw 1\circ b) \diamond S\sigma(a\sw 2);
\end{equation}
\end{subequations}
cf.\  \eqref{lam.sig.Hopf} and \eqref{Hopf.act.mu}.
\item For all $a,b\in A$,
\begin{subequations}
 \begin{equation}\label{coc.Hopf}
 \sigma(a\circ b) = \sum \sigma(a\sw 1)\diamond \left(a\sw 2\la \sigma(b)\right),
 \end{equation}
 \begin{equation}\label{coc.Hopf.mu}
 \sigma(a\circ b) = \sum  \left(a\sw 1\bla \sigma(b)\right)\diamond \sigma(a\sw 2).
 \end{equation}
 \end{subequations}
\end{zlist}
\end{thm}
\begin{proof}
(1)  We need to prove that, for all $a,b,c\in A$,
\begin{subequations}\label{act.Hopf}
\begin{equation}\label{act.Hopf1}
a\la (b\diamond c) = \sum (a\sw 1 \la b)\diamond  (a\sw 2\la c), \qquad a\la 1_\diamond = \eps(a)1_\diamond,
\end{equation}
\begin{equation}\label{act.Hopf2}
(a\circ b)\la c = a\la\left(b\la c\right), 
\end{equation}
\end{subequations}
and similar equalities for $\bla$.

Properties \eqref{act.Hopf1} are verified by direct calculations. For all $a,b,c\in A$,
\begin{align*}
a\la(b\diamond c) &= \sum S\sigma(a\sw 1) \diamond  (a\sw 2\circ (b\diamond c)) \\
&= \sum  S\sigma(a\sw 1) \diamond (a\sw 2\circ b)\diamond S\sigma(a\sw 3) \diamond (a\sw 4 \circ c))\\
&= \sum (a\sw 1 \la b)\diamond (a\sw 2\la c),
\end{align*}
where the compatibility condition \eqref{brace.Hopf} has been used. By the fact that $\sigma$ is a coalgebra morphism and using equation \eqref{sigma.Hopf} in Lemma~\ref{lem.coc.prop.Hopf} one easily finds
$$
a\la 1_\diamond = \sum S\sigma(a\sw 1)\diamond (a\sw 2\circ 1_\diamond)  = S\sigma(a\sw 1)\diamond \sum \sigma(a\sw 2) = \eps(a)1_\diamond.
$$
Finally,
\begin{align*}
a\la (b\la c) &= \sum  S\sigma(a\sw 1) \diamond \left(a\sw 2\circ \left(S\sigma(b\sw 1) \diamond \left(b\sw 2\circ c\right)  \right) \right)\\
&= \sum  S\sigma(a\sw 1) \diamond \sigma(a\sw 2) \diamond S(a\sw 3 \circ \sigma(b\sw 1)) \diamond (a\sw 4\circ b\sw 2\circ c)\\
&= \sum  S\sigma(a\sw 1 \circ b\sw 1) \diamond (a\sw 2\circ b\sw 2\circ c) = (a\circ b)\la c,
\end{align*}
where we used equation \eqref{Sb.c} in Lemma~\ref{lem.diamond.antipode} to derive the second equality, the equivariance \eqref{equivariant} and the coalgebra morphism  properties of $\sigma$ to derive the third equality, and  the bialgebra and Hopf algebra axioms elsewhere

The statement concerning the other action $\bla$  is also proven by direct calculations, similar to that for $\la$.  Explictly, for all $a,b,c\in A$,
\begin{align*}
a\bla(b\diamond c) &= \sum (a\sw 1\circ (b\diamond c))\diamond S\sigma(a\sw 2) \\
&= \sum (a\sw 1\circ b)\diamond S\sigma(a\sw 2) \diamond (a\sw 3 \circ c))\diamond S\sigma(a\sw 4)\\
&= \sum (a\sw 1 \bla b)\diamond (a\sw 2\bla c),
\end{align*}
where the compatibility condition \eqref{brace.Hopf} has been used. By the fact that $\sigma$ is a coalgebra morphism and using  \eqref{sigma.Hopf} one easily finds
$$
a\bla 1_\diamond = \sum (a\sw 1\circ 1_\diamond)\diamond S\sigma(a\sw 2) = \sum \sigma(a\sw 1)\diamond S\sigma(a\sw 2) = \eps(a)1_\diamond.
$$
Finally,
\begin{align*}
a\bla (b\bla c) &= \sum \left(a\sw 1\circ \left(\left(b\sw 1\circ c\right) \diamond S\sigma(b\sw 2)\right) \right) \diamond S\sigma(a\sw 2)\\
&= \sum (a\sw 1\circ b\sw 1\circ c)\diamond S(a\sw 2 \circ \sigma(b\sw 2))\diamond \sigma(a\sw 3)\diamond S\sigma(a\sw 4)\\
&= \sum (a\sw 1\circ b\sw 1\circ c)\diamond S\sigma(a\sw 2 \circ b\sw 2) = (a\circ b)\bla c,
\end{align*}
where we used Lemma~\ref{lem.diamond.antipode} to derive the second equality, then \eqref{equivariant} and the fact that $\sigma$ is a coalgebra morphism  to derive the third equality, and  the bialgebra and Hopf algebra axioms elsewhere. Hence also $\bla$ is an action compatible with the bialgebra structure of $(A,\circ)$.

(2) Starting from the left hand side of \eqref{coc.Hopf} and using \eqref{lam.sig.Hopf}, \eqref{sig.endo} and \eqref{equivariant} we compute
  \begin{align*}
 \sum \sigma(a\sw 1)\diamond \left(a\sw 2\la \sigma(b)\right) &= \sum \sigma(a\sw 1)\diamond S\sigma(a\sw 2) \diamond  \left(a\sw 3\circ \sigma(b)\right)\\
 &= a\circ \sigma(b) =  \sigma(a\circ b),
  \end{align*}
  as required.
  Similarly,
   \begin{align*}
 \sum  \left(a\sw 1\bla \sigma(b)\right) \diamond  \sigma(a\sw 2)&= \sum  \left(a\sw 1\circ \sigma(b)\right) \diamond S\sigma(a\sw 2)\diamond \sigma(a\sw 3)  \\
 &= a\circ \sigma(b) =  \sigma(a\circ b).
  \end{align*}
This completes the proof of the theorem.
\end{proof}

 \begin{rem}\label{rem.coc}
 Recall e.g.\ from \cite{Swe:coh} that \eqref{coc.Hopf} means that the map $\sigma: A\to A$ given by \eqref{sigma.Hopf} is a one-cocycle on the bialgebra $(A,\circ)$ with values in a Hopf module algebra $(A,\diamond)$. This justifies the adopted terminology for $\sigma$.
 
 Note also that if $(A,\Delta,\circ)$ is a unital bialgebra (with identity $1_\circ$), then $\sigma(1_\circ) = 1_\diamond$ by Lemma~\ref{lem.coc.prop.Hopf} and, since $1_\circ$ is a group-like element, both actions $\la$ and $\bla$ are unital.
 \endeg
 \end{rem}

 \begin{prop}\label{prop.hierarchy.Hopf}
 Let $(A,\Delta, \diamond,\circ)$ be a Hopf truss with cocycle $\sigma$. For any grouplike element $e$ in $(A,\Delta, \diamond)$, define a binary operation $\diamond_e$ on $A$ by
 \begin{equation}\label{diam.e.Hopf}
 a\diamond_e b = a\diamond S(e)\diamond b.
 \end{equation}
 Then $(A,\Delta, \diamond_e)$ is a Hopf algebra with identity $e$ and the antipode
 \begin{equation}\label{antipode.n}
 S_e(a) = e\diamond S(a)\diamond e.
 \end{equation}
 Furthermore, $(A,\Delta, \diamond_e,\circ)$ is a Hopf truss with cocycle $\sigma_e(a) = a\circ e$.
 \end{prop}
\begin{proof}
Since $e$ is a grouplike element it is a unit in $(A,\diamond)$ and $S(e) = e^\diamond$. Thus the formula \eqref{diam.e.Hopf} is in fact identical with that in Corollary~\ref{cor.family}. One easily checks that $(A,\Delta, \diamond_e)$  is a Hopf algebra (isomorphic to $(A,\Delta, \diamond)$) with the antipode and identity as described. It remains to check the compatibility condition \eqref{brace.Hopf}. To this end let us take any $a,b,c\in A$ and, using equation \eqref{Sb.c} in Lemma~\ref{lem.diamond.antipode}, properties of the Hopf truss cocycle listed in Lemma~\ref{lem.coc.prop.Hopf}, the fact that $\sigma$ is a coalgebra map and properties of the antipode, compute
\begin{align*}
a\circ(b\diamond_e c) &= a\circ (b\diamond S(e)\diamond c) \\
&= \sum (a\sw 1 \circ b)\diamond S\sigma(a\sw 2)\diamond \left(a\sw 3\circ \left(S(e)\circ c\right)\right)\\
&= \sum (a\sw 1 \circ b)\diamond S\sigma(a\sw 2)\diamond \sigma(a\sw 3) \diamond S(a\sw 4 \circ e) 
\diamond  (a\sw 5\circ c)\\
&= \sum (a\sw 1 \circ b)\diamond S(e)\diamond e \diamond S\sigma_e(a\sw 2 )
\diamond e\diamond S(e)\diamond (a\sw 3\circ c)\\
& = \sum (a\sw 1 \circ b)\diamond_e S_e(\sigma_e(a\sw 2 )) \diamond_e  (a\sw 3\circ c).
\end{align*}
Therefore, $(A,\Delta, \diamond_e,\circ)$ is a Hopf truss, as stated. 
\end{proof}

In particular, since $\sigma$ is a coalgebra map and $1_\diamond$ is a grouplike element, we can take $e = \sigma^n(1_\diamond)$  in Proposition~\ref{prop.hierarchy.Hopf}, in which case we obtain a Hopf truss with cocycle $\sigma_e = \sigma^{n+1}$. Furthermore, if $(A,\Delta,\circ)$ is a unital bialgebra, then setting $e = \sigma(1_\circ)$ we obtain $\sigma_e = \id$, and hence operations in $(A,\Delta, \diamond_e,\circ)$ satisfy the Hopf brace-type distributive law

\begin{prop}\label{prop.Hopf.morph}
Let $(A,\Delta_A, \diamond, \circ)$ and   $(B,\Delta_B, \diamond, \circ)$ be Hopf trusses with cocycles $\sigma_A$ and $\sigma_B$, respectively. Let $\la$ and $\bla$ be left actions defined in Theorem~\ref{thm.truss.Hopf}. A linear map $f:A\to B$ that is a homomorphism of respective Hopf algebras and bialgebras has the following properties: 
\begin{rlist}
\item For all $a\in A$
\begin{equation}\label{diag.sig}
f(\sigma_A(a)) = \sigma_B(f(a)).
\end{equation}
\item For all $a,a'\in A$,
\begin{equation}\label{diag.lam}
f(a\la a') = f(a)\la f(a').
\end{equation}
\item
For all $a,a'\in A$,
\begin{equation}\label{diag.mu}
f(a\bla a') = f(a)\bla f(a').
\end{equation}
\end{rlist}

Furthermore if the antipode $S_B$ of $(B,\Delta_B, \diamond)$ is injective, then all these properties are mutually equivalent.
\end{prop}
\begin{proof}
Since $\sigma$ is fully determined by $\circ$ and $1_\diamond$, precisely in the same way as in the skew truss case, equation \eqref{diag.sig} follows by the preservation properties of $f$ by the same arguments as in the proof of Proposition~\ref{prop.morph}.

Using \eqref{diag.sig} and the fact that $f$ is a Hopf algebra map we can compute, for all $a,a'\in A$,
\begin{align*}
f(a)\la f(a') &= \sum S_B\left(\sigma_B\left(f(a\sw 1)\right)\right)\diamond f(a\sw 2\circ a')\\
& = \sum S_B\left(f\left(\sigma_A\left(a\sw 1\right)\right)\right)\diamond f(a\sw 2\circ a')\\
& = \sum f\left(S_A\left(\sigma_A\left(a\sw 1\right)\right)\diamond (a\sw 2\circ a')\right) = f\left(a\la a'\right) .
\end{align*}
In particular, a Hopf algebra and a bialgebra morphism has property (ii).

Coversely, using the above arguments one easily sees that \eqref{diag.lam}  implies that
\begin{align*}
\sum S_B\left(\sigma_B\left(f(a\sw 1)\right)\right) & \diamond f(a\sw 2\circ a'\sw 1)  \ot f(a\sw 3\circ a'\sw 2)\\
& = \sum S_B\left(f\left(\sigma_A\left(a\sw 1\right)\right)\right)\diamond f(a\sw 2\circ a'\sw 1) \ot f(a\sw 3\circ a'\sw 2).
\end{align*}
Applying $\diamond(\id\ot S_B)$ to both sides of this equality and using the fact that $f$ is a coalgebra homomorphism one obtains 
$$
S_B\left(f(\sigma_A(a))\right) = S_B\left(\sigma_B(f(a))\right),
$$
and  the injectivity of $S_B$ yields
\eqref{diag.sig}. Hence (ii) implies (i).

The implication (i) $\implies$ (iii) and its converse (for the injective $S_B$) are  shown in a similar way.
\end{proof}
 
A Hopf algebra and a bialgebra map $f$ (thus satisfying equivalent conditions of Proposition~\ref{prop.Hopf.morph}) is a {\em morphism of Hopf trusses}.
As was the case for skew trusses also Hopf truss structures can be ported by isomorphisms.
 \begin{lem}\label{lem.ind.Hopf}
 Let $(A,\Delta,\diamond,\circ)$ be a Hopf  truss with cocycle $\sigma$.
 \begin{zlist}
 \item  Let $(B,\ast, \Delta_B)$ be a Hopf algebra. Given an isomorphism of Hopf algebras $f:(B,\ast,\Delta_B)\to (A,\diamond,\Delta)$, define a binary operation on $B$ by porting $\circ$ to $B$ through $f$, i.e.\
 $$
 a\circ_f b = f^{-1}\left(f(a)\circ f(b)\right), \qquad \mbox{for all $a,b\in B$}.
 $$
 Then $(B,\Delta_B,\ast,\circ_f)$ is a Hopf truss isomorphic to $(A,\Delta,\diamond,\circ)$, with  cocycle $\sigma_B := f^{-1}\sigma f$.
 \item   Let $(B,\Delta_B, \bullet)$ be a bialgebra and let $g:(B,\Delta_B, \bullet)\to (A,\Delta,\circ)$ be  an isomorphim of bialgebras. Define a binary operation on $B$ by porting $\diamond$ to $B$ through $g$, i.e.\
 $$
 a\diamond_g b = g^{-1}\left(g(a)\diamond g(b)\right), \qquad \mbox{for all $a,b\in B$}.
 $$
 Then $(B,\Delta_B, \diamond_g,\bullet)$ is a Hopf truss isomorphic to $(A,\Delta, \diamond,\circ)$, with  cocycle $\sigma_B := g^{-1}\sigma g$.
 \end{zlist}
 \end{lem}
\begin{proof}
This lemma is proven by direct calculations, obtained as linearisation of calculations in the proof of Lemma~\ref{lem.ind}. Since we have displayed calculation proving statement (1) of Lemma~\ref{lem.ind}, here we display  calculations proving statement (2).

Since $g$ is a bialgebra homomorphism (i.e.\ it respects both algebra and coalgebra structures) one quickly deduces that $B$ with product $\diamond_g$, identity $1_{\diamond_g} = g^{-1}(1_\diamond)$ and the original coproduct $\Delta_B$ is a unital bialgebra. If $S$ denotes the antipode of the Hopf algebra $(A,\diamond, \Delta)$, then $S_B = g^{-1}Sg$ is the antipode of $(B, \diamond_g,\Delta_B)$. Indeed, for all $b\in B$,
\begin{align*}
\sum b\sw 1\diamond_g S_B(b\sw 2) &= \sum g^{-1}\left(g(b\sw 1)\diamond S(g(b\sw 2))\right)\\
& = \sum g^{-1}\left(g(b)\sw 1\diamond S(g(b)\sw 2)\right)
= g^{-1}(\eps(g(b))1_\diamond) = \eps_B(b)1_{\diamond_g},
\end{align*}
and similarly for the second antipode condition. Finally, the Hopf truss compatibility condition \eqref{brace.Hopf} can be checked as follows:
\begin{align*}
a\bullet (b\diamond_g c) &= g^{-1}\left(g\left(a\right)\right)\bullet g^{-1}\left(g\left(b\right)\diamond g\left(c\right)\right) = g^{-1}\left(g\left(a\right)\bullet \left(g\left(b\right)\diamond g\left(c\right)\right)\right)\\
&= \sum  g^{-1}\left(\left(g\left(a\right)\sw 1\bullet g\left(b\right)\right)\diamond S(\sigma(g\left(a\right)\sw 2)) \diamond \left(g\left(a\right)\sw 3\bullet g\left(c\right)\right)\right)\\
&= \sum  g^{-1}\left(\left(g\left(a\sw 1\right)\bullet g\left(b\right)\right)\diamond S(\sigma(g\left(a\sw 2\right))) \diamond \left(g\left(a\sw 3\right)\bullet g\left(c\right)\right)\right)\\
&=\sum \left(a\sw 1\bullet b\right)\diamond_g g^{-1} \left(S(\sigma(g\left(a\sw 2\right))) \right)\diamond_g\left(a\sw 3\bullet c\right)\\
&=\sum \left(a\sw 1\bullet b\right)\diamond_g S_B\left(\sigma_B\left(a\sw 2\right)\right) \diamond_g\left(a\sw 3\bullet c\right),
\end{align*}
as required.
\end{proof}

We finish with a few comments as to when a Hopf truss becomes a Hopf brace. Since the formula \eqref{sigma.Hopf} for the cocycle $\sigma$ of a  Hopf truss is the same as \eqref{sigma},  $\sigma$ is bijective if $1_\diamond$ is an invertible element of $(A,\circ)$. 

\begin{lem}\label{lem.Hopf.brace}
Let  $(A,\Delta, \diamond,\circ)$ be a  Hopf truss with cocycle $\sigma$.  If
$(A,\Delta, \circ)$ is a Hopf algebra, then $A$ is a Hopf brace with respect to $\diamond$ and the operation $\bullet$ defined by
\begin{equation}\label{bullet.Hopf}
a\bullet b = \sigma^{-1}(a)\circ b = a\circ 1_\diamond^\circ\circ b, \qquad \mbox{for all $a,b\in A$}.
\end{equation}
\end{lem}
\begin{proof}
Since $1_\diamond$ is a group-like element it is invertible in the Hopf algebra  $(A,\Delta,\circ)$, and hence the map  $\sigma$ is bijective by Lemma~\ref{lem.sigma.inverse}. Hence, by Theorem~\ref{thm.truss.Hopf}, $\sigma$ is a bijective one-cocycle on the Hopf algebra  $(A,\circ)$ with values in the Hopf algebra $(A,\diamond)$. By \cite[Theorem~1.12]{AngGal:Hop}, $A$ is a Hopf brace with operations $\diamond$ and the induced operation
$$
a\bullet b = \sigma\left(\sigma^{-1}(a)\circ\sigma^{-1}(b)\right) = \sigma^{-1}(a)\circ b = a\circ 1_\diamond^\circ\circ b,
$$
by the same arguments as in the proof of Lemma~\ref{lem.brace}.
\end{proof}

\section{Concluding remarks}
The aim of this paper was to obtain some insight into the algebraic nature of the brace distributive law. As a starting point we introduced a law interpolating that of ring and brace distributive laws and an algebraic system bound by this law, which we termed truss. Braces arose from the study of the Yang-Baxter equation, but they proved to be an extremely worthwhile object of study on their own right, since their applications reach far beyond the motivating problem, deep into group theory. Trusses arise from the study of braces; the study of trusses alone might be a worthwhile task too. Should this be the case, there would be many questions that could be posed and answers to which should be sought. Here we mention but a few. Is the notion of a pith a good replacement for that of a kernel? What can play the role of an ideal of a truss? What is its socle? What is a module over a truss? How to construct and what can be said about the quotient of a truss by its ideal? What can be said about the category of trusses: what properties does it have, what universal constructions does it admit?  Should one stick to morphisms as defined in the present paper, or should one define them using the heap structure as in Remark~\ref{rem.heap.morph}? What is the structure of category defined by this means? 
Can the heap formulation of solutions to the set-theoretic Yang-Baxter equation in Remark~\ref{rem.YB} help in constructing new classes of solutions? What can be said about trusses for general heaps 
as proposed in Remark~\ref{rem.heap}, without reference to the inherent group structure? 

As with any possibly novel, but closely related to existing, algebraic structure the above list of questions can be easily extended. The author of these notes hopes that at least some of the questions posed will be considered worth addressing.

\subsection*{Acknowledgments}
I would like to thank Eric Jespers for a discussion  in May 2017, which encouraged me to present my thoughts and observations in a systematic way. I would also like to express my gratitude to the participants (and the organisers) of the Workshop on Non-commutative and Non-associative Algebraic Structures in Physics and Geometry (Belfast, August 2017), where some of the results of this paper were presented, for their interest and very helpful comments. I would also like to thank Michael Kinyon, for sharing with me his observations on the very natural interpretation of the truss distributive law in terms of heaps. Finally I would like to thank Charlotte Verwimp, whose comment on the pith of a skew ring morphism led me to rethinking and redesigning the notion of a pith.

The research presented in this paper is partially supported by the Polish National Science Centre grant 2016/21/B/ST1/02438.

\end{document}